\RequirePackage{etex}
\documentclass{siamart}

\usepackage{amssymb}
\usepackage{amsmath}
\usepackage{xcolor}
\usepackage{cleveref}
\usepackage{booktabs}
\usepackage[binary-units]{siunitx}

\usepackage{pgfplots}
\usepackage{tikz}
\usepgflibrary{arrows}
\usepackage{graphicx}
\usetikzlibrary{decorations.pathmorphing}

\usetikzlibrary{external} %

\tikzsetexternalprefix{figures/}
\tikzexternaldisable

\makeatletter
\newcommand{\useexternalfile}[1]{%
  \tikzsetnextfilename{#1}%
  \tikzexternalenable
  \input{\tikzexternal@filenameprefix#1.tikz}
  \tikzexternaldisable
}\makeatother

\allowdisplaybreaks

\usepackage{enumitem}
\renewcommand{\descriptionlabel}[1]{\hspace{\labelsep}\emph{#1}:}

\usepackage{notation}
\usepackage{poolingnotation}

\usepackage{multicol}
\usepackage{multirow}

\usepackage{todonotes}

\usepackage{xspace}

\newcommand{\ext}{\operatorname{ext}}

\newcommand{\exclude}[1]{}

\newcommand{\ogamma}{\overline{\gamma}}
\newcommand{\ugamma}{\underline{\gamma}}
\newcommand{\oGamma}{\overline{\beta}}
\newcommand{\uGamma}{\underline{\beta}}
\newcommand{\alphep}{\alpha^{\epsilon}}

\newcommand*\samethanks[1][\value{footnote}]{\footnotemark[#1]}

\title{Strong Convex Nonlinear Relaxations of the Pooling Problem}
\date{\today}
\author{James Luedtke\thanks{Department of Industrial and Systems Engineering, University of Wisconsin-Madison, Madison, WI, USA    
    } \and Claudia D'Ambrosio\thanks{LIX CNRS (UMR7161), \'{E}cole
    Polytechnique, 91128 Palaiseau Cedex, France
    } \and Jeff Linderoth\samethanks[1]
    \and Jonas Schweiger\thanks{Zuse Institute Berlin (ZIB), Takustrasse 7,
      14195 Berlin, Germany, ORCID 0000-0002-4685-9748 
    }
    }

\begin{document}

\maketitle

\begin{abstract}
We investigate new convex relaxations for the pooling problem, a classic nonconvex
production planning problem in which input materials are mixed in
intermediate pools, with the outputs of these pools further mixed to
make output products meeting given attribute percentage requirements. 
Our relaxations are derived by considering a set which arises from the formulation by considering a single product, a
single attibute, and a single pool. The convex hull of the resulting nonconvex set is not polyhedral. We derive valid
linear and convex nonlinear inequalities for the convex hull, and demonstrate that different subsets of these
inequalities define the convex hull of the nonconvex set in three cases determined by the parameters of the set.
Computational results on literature instances and newly created larger test instances demonstrate that the inequalities
can significantly strengthen the convex relaxation of the $pq$-formulation
of the pooling problem, which is the relaxation known to have the strongest bound.
\end{abstract}

\section{Introduction}\label{sec:introduction}
The pooling problem is a classic \nonconvex nonlinear problem
introduced by Haverly in 1978 \cite{haverly:78}.
The problem consists in routing flow through a feed forward network
from inputs through pools to output products. The material that is introduced
at inputs has known quality for certain attributes. The task is to find
a flow distribution that respects quality restrictions on the output products.
As is standard in the pooling problem, we assume linear blending, \ie the attributes at a node are mixed in
the same proportion as the incoming flows. As the quality of the
attributes in the pools is dependent on the decisions determining amount of flow from inputs to the pools, 
the resulting constraints include bilinear terms.

The aim of this work is to strengthen the relaxation of the strongest known formulation,
i.e., the so-called \pqformulation proposed in
\cite{Quesada_pooling_1995, tawarmalani.sahinidis:02-1}. By focusing on a single output product, a single attribute, and
a single pool, and aggregating variables, we derive a
structured \nonconvex 5-variable set that is a relaxation of the original feasible
set. The description of this set contains one bilinear term which
captures some of the \nonconvex essence of the problem. Valid convex inequalities for this set directly
translate to valid inequalities for the original pooling problem. We derive valid linear and nonlinear convex
inequalities for the set. For three cases determined by the parameters of the set, we demonstrate that a subset of these inequalities define the convex hull of the set.
Finally, we conduct an illustrative computational study that demonstrates that these inequalities can indeed improve the
relaxation quality over the \pqformulation, particularly on instances in which the underlying network is sparse, which
are precisely the instances in which the \pqformulation relaxation has the largest gap to the optimal value.  As part of
this study, we create and test the inequalities on new, larger test instances of the pooling problem. 

The reminder of the paper is organized as follows. We briefly review relevant literature on the pooling problem in the
remainder of this section. In
\cref{sec:classic_formulations}, we introduce the \pqformulation, and its classic relaxation 
based on the McCormick relaxation. Our set of interest, that
represents a relaxation of the pooling problem, is introduced in
\cref{sec:strong_convex_nonlinear_relaxations}. In the same section,
we present the valid inequalities for this set. We then prove in
\cref{sec:convex_hull_analysis} that 
certain subsets of the proposed inequalities define the convex hull of the set of interest, for three cases based
on the parameters of the set.
Computational results are presented in
\cref{sec:computational_results}, and concluding remarks are made in \cref{sec:conclusions}.

\subsection{Literature Review}
\label{sec:literature}

There are many applications of the pooling problem, including
petroleum refining, wastewater treatment, and general chemical
engineering process design \cite{B00,D89,K00,R95}. This is confirmed
by an interesting analysis performed by Ceccon et al.
\cite{ceccon_kouyialis_misener} whose method allows to recognize
pooling problem structures in general mixed integer nonlinear
programming problems.

Although the pooling problem has been studied for decades, it was only proved to be strongly NP-hard in 2013 by Alfaki and
Haugland \cite{alfaki.haugland:13}.
Further complexity results on special cases of the pooling problem
can be found in \cite{Baltean-Lugojan2017,Boland2017,Haugland2015}.

Haverly \cite{haverly:78} introduced the pooling problem using what is now known as the 
\pformulation. Almost 20 years later, Ben-Tal et al.
\cite{ben-tal.eiger.gershovitz:94} proposed an equivalent formulation
called \qformulation. Finally, the \pqformulation was introduced in
\cite{Quesada_pooling_1995, tawarmalani.sahinidis:02-1} and is a
strengthening of the \qformulation. It has been shown to be the
strongest known formulation for the pooling problem \cite{tawarmalani.sahinidis:02-1}; i.e., the ``natural
relaxation'' of this formulation \pqformulation yields a bound on the optimal value that is at least as good as that of any other known formulation. 

Many other approaches for solving the pooling problem have been proposed, including: recursive and successive linear
programming \cite{baker.lasdon:85,haverly:78}, decomposition methods
\cite{floudas.aggarwal:90}, and global optimization
\cite{foulds.haugland.jornsten:92}. More recently, 
Dey and Gupte \cite{DeyGupteOR2015} used discretization strategies to
design an approximation algorithm and an effective heuristic. 
Several variants of the standard pooling problem have been
studied, see, for example, \cite{audet.et.al:04,
  misener.floudas:09, Misener20101432, Ruiz2012}. Some of the variants
introduce binary variables to model design decisions, thus
yielding a mixed-integer nonlinear programming problem, see, for
example, \cite{dambrosio.linderoth.luedtke:11, meyer.floudas:06:pooling,
  misener_floudas_2010, Visweswaran2001}. For
more comprehensive reviews of the pooling problem the reader is referred to
\cite{gupte:12,Gupteetal_poolingproblemoverview,misener.floudas:09,tawarmalani.sahinidis:02-1} and to \cite{Gupte2017} for an overview on the relaxations and discretizations for the pooling problem.

{\bf Notation.} For a set $T$, $\conv\Of{T}$ denotes the convex hull of $T$, and for a convex set $R$, 
$\ext\Of{R}$ denotes the set of extreme points of $R$.


\section{Mathematical Formulation and Relaxation}
\label{sec:classic_formulations}

There are multiple formulations for the pooling problem,
primarily differing in the modeling of the concentrations of
attributes throughout the network. We base our work on the
state-of-the-art \emph{\pqformulation}.

We are given a directed graph
$\graph = (\nodes, \arcs)$ where $\nodes$ is the set of vertices that
is partitioned into inputs
\inputs,  pools \pools, and outputs \outputs, \ie $\nodes = \inputs
\cup \pools \cup \outputs$.
For a node $\node\in\nodes$, the sets $\inputs[\node]\subseteq\inputs$, $\pools[\node]\subseteq\pools$,
$\outputs[\node]\subseteq\outputs$ denote the inputs, pools, and outputs, respectively, that are
directly connected to \node. Arcs $(i, j) \in \arcs$ link
inputs to pools, pools to outputs, and inputs directly to outputs, \ie
$\arcs \subseteq (\inputs\times\pools) \cup (\pools\times\outputs) \cup (\inputs\times\outputs)$. In
particular, pool-to-pool connections are not considered.

The \pqformulation of the pooling problem uses the following decision variables:
\begin{itemize}
\item $\flowOnArc{i}{j}$ is the flow on $(i,j)\in A$;
\item $\proportionVar{\inp}{\pool}$ is the proportion of flow to pool
  $\pool \in \pools$ that comes from input $\inp \in \inputs[\pool]$;
\item \flowpath is the flow from $\inp\in\inputs$
  through pool $\pool \in \pools[\inp]$ to output $\outp \in \outputs[\pool]$.
\end{itemize}
With these definitions, the \pqformulation of the
pooling problem is:
\begin{subequations}\label{eq:pooling-region}
\begin{align}
  \min \quad & \sum_{(i,j) \in \arcs} c_{ij} \flowOnArc{i}{j} \label{eq:pq_obj}\\
  \st \quad  & \sum_{\pool \in \pools[\inp]} \flowOnArc{\inp}{\pool} +
               \sum_{\outp \in \outputs[\inp]} \flowOnArc{\inp}{\outp} \leq
               \capacity{\inp} %
             && \fa \inp \in \inputs \label{eq:pq_input-capacity}\\
             & \sum_{\outp \in \outputs[\pool]}
              \flowOnArc{\pool}{\outp} \leq \capacity{\pool}
             && \fa \pool\in\pools \label{eq:pq_pool-capacity}\\
             &\sum_{\pool \in \pools[\outp]} \flowOnArc{\pool}{\outp} +
               \sum_{\inp \in \inputs[\outp]} \flowOnArc{\inp}{\outp} \leq
               \capacity{\outp} %
             && \fa \outp \in \outputs \label{eq:pq_product-capacity}\\
             & \sum_{\inp \in \inputs[\pool]}
               \proportionVar{\inp}{\pool} = 1 &&\fa \pool
                            \in\pools \label{eq:pq_q-proportion}\\
 & \flowpath = \proportionVar{\inp}{\pool} \flowOnArc{\pool}{\outp}  && \fa \inp \in \inputs[\pool],\ \pool \in \pools[\outp],\ \outp
                                                                           \in \outputs\label{eq:pq_wdef}\\
             & \flowOnArc{\inp}{\pool} = \sum_{\outp \in \outputs[\pool]}
               \flowpath &&\fa \inp\in\inputs[\pool],\ \pool
                            \in\pools \label{eq:pq_flow-balance}\\
           & \sum_{\inp \in \inputs[\outp]} \ogammaIdx
               \flowOnArc{\inp}{\outp} + \sum_{\pool \in
               \pools[\outp]}\sum_{\inp \in \inputs[\pool]} \ogammaIdx
               \flowpath \leq 0 && \fa \outp\in\outputs,\
                                   \spec\in\specs \label{eq:pq_quality-upper}\\
  & \sum_{\inp \in \inputs[\pool]}
    \flowpath = \flowOnArc{\pool}{\outp} &&  \fa \outp\in\outputs[\pool],\ \pool \in\pools \label{eq:pq_pq-1}\\
    & \sum_{\outp \in \outputs[\pool]}
  \flowpath \leq \capacity{\pool}\proportionVar{\inp}{\pool} &&
  \fa \inp\in\inputs[\pool],\ \pool\in\pools \label{eq:pq_pq-2}\\
   & 0 \leq \flowOnArc{i}{j} \leq \capacity{ij} && \fa (i,j) \in \arcs \label{eq:pq_pq_xub}\\
             & 0\leq \proportionVar{\inp}{\pool} \leq 1 && \fa \inp\in\inputs[\pool],\
                                                     \pool\in\pools \label{eq:pq_qgeq0}.
\end{align}
\end{subequations}
%

The objective \cref{eq:pq_obj} is to minimize the production cost, where $c_{ij}$ is the cost per  unit flow on arc
$(i,j)$.
Inequalities
\cref{eq:pq_input-capacity,eq:pq_pool-capacity,eq:pq_product-capacity}
represent capacity constraints on inputs, pools, and outputs,
respectively, where here $\capacity{\inp}, \inp \in \inputs$, $\capacity{\pool},\pool\in\pools$, and $\capacity{\outp},
\outp \in \outputs$ are given capacity limits.  Equations \cref{eq:pq_q-proportion} enforce that the
proportions at each pool sum up to one. Equations
\cref{eq:pq_wdef,eq:pq_flow-balance} define the auxiliary variables
$\flowpath$ and link them to the flow variables.
\eqref{eq:pq_quality-upper} formulates the quality constraints for
each attribute $\spec$ in the set of attributes $\specs$. The
coefficients $\ogammaIdx$ represent the excess of attribute quality
$\spec$ of the material from input \inp{} with respect to the upper quality bound
at output $\outp$. The upper quality bound is met when there is no
excess, \ie the total excess is not positive. For brevity, we do not
include lower bounds on attribute quality at the final products, but
these can be easily added. Inequalities \cref{eq:pq_pq-2,eq:pq_pq-1}
are redundant in the formulation but are not when the \nonconvex
constraints \cref{eq:pq_wdef} are not enforced as is done in a
relaxation-based solution algorithm. These two constraints constitute
the difference between the $q$- and the \pqformulation and are
responsible for the strong linear relaxation of the latter.
Finally, \eqref{eq:pq_pq_xub} limits the flow on each arc $(i,j)$ to a given capacity $C_{ij}$.

%
A linear programming relaxation of the \pqformulation is obtained by
replacing the constraints \eqref{eq:pq_wdef} with the McCormick
inequalities derived using the bounds \eqref{eq:pq_pq_xub} and
\eqref{eq:pq_qgeq0}:
\begin{subequations}
\label{eq:mcall}
\begin{alignat}{2}
 &\flowpath \leq \flowOnArc{\pool}{\outp}, \ \flowpath \leq
  \capacity{\pool\outp} \proportionVar{\inp}{\pool}, \  &\quad& \fa \inp \in \inputs[\pool],\ \pool \in \pools[\outp],\ \outp
                                                                           \in \outputs  \label{eq:mcall1} \\
 & \flowpath \geq 0, \
  \flowpath \geq \capacity{\pool\outp} \proportionVar{\inp}{\pool} +
  \flowOnArc{\pool}{\outp} - \capacity{\pool\outp}, &\quad& \fa \inp \in \inputs[\pool],\ \pool \in \pools[\outp],\ \outp
                                                                           \in \outputs . \label{eq:mcall2}
\end{alignat}
\end{subequations}
We refer to the relaxation obtained by replacing \eqref{eq:pq_wdef} with \eqref{eq:mcall} as the {\it McCormick
relaxation}.  Our goal is to derive tighter relaxations of the pooling problem by considering more of the problem
structure.


\section{Strong Convex Nonlinear Relaxations}\label{sec:strong_convex_nonlinear_relaxations}
To derive a stronger relaxation of the pooling problem, we seek to identify a relaxed set that contains the feasible region of
the pooling problem, but includes some of the key \nonconvex structure. First, we consider only one single attribute
$\spec \in \specs$ and relax all constraints
\eqref{eq:pq_quality-upper} concerning the other attributes. Next, we
consider only one output $\outp \in \outputs$, and remove all nodes and
arcs which are not in a path to output $\outp$. In particular, this
involves all other outputs. Then, we focus on pool $\pool \in \pools$
with the intention to split flows into two categories: the flow
through pool $\pool$ and aggregated flow on all paths not passing
through pool $\pool$, also called ``by-pass'' flow. Finally, we
aggregate all the flow from the inputs to pool $\pool$.

As a result, we are left with five decision variables:
\begin{enumerate}
  \item the total flow through the pool, \ie the flow $x_{\pool \outp}$ from pool $\pool$ to output $\outp$
  \item the total flow $z_{\pool \outp}$ over the by-pass, \ie the
    flow to output $\outp$ that does not pass through pool $\pool$
       $$z_{\pool \outp} := \sum_{\inp \in \inputs[\outp]} \flowOnArc{\inp}{\outp} + \sum_{\pool' \in \pools[\outp] \mid \pool' \neq \pool}
  \flowOnArc{\pool'}{j}$$
    \item the contribution $u_{\spec \pool \outp}$ of the flow through pool $\pool$ to the excess of attribute $\spec$ at output $\outp$, i.e.,
    $$u_{\spec \pool\outp} := \sum_{\inp \in \inputs[\pool]} \gamma_{\spec \inp \outp} w_{\inp \pool\outp}$$
  \item the contribution $y_{\spec \pool \outp}$ of by-pass flow to the excess of attribute $\spec$ at output $\outp$, i.e.,
  $$y_{\spec \pool \outp} := \sum_{\inp \in \inputs[\outp]} \gamma_{\spec \inp \outp}
    \flowOnArc{\inp}{\outp} + \sum_{\pool' \in \pools[\outp] \mid \pool' \neq \pool} \sum_{\inp \in \inputs[\pool']} \gamma_{\spec \inp \outp} w_{\inp \pool' \outp}$$
    \item the attribute quality $t_{\spec \pool \outp}$ of the flow through pool $\pool$, i.e.,
    $$t_{\spec \pool} := \sum_{\inp \in \inputs[\pool]} \gamma_{\spec \inp \outp} q_{\inp \pool}.$$
\end{enumerate}
With these decision variables, the quality constraint associated with
attribute $\spec$ of output $\outp$ and the capacity constraint
associated with output $\outp$ from \eqref{eq:pooling-region} can be
written as:
\begin{subequations}
\label{eq:newvarcons}
\begin{align}
y_{\spec\pool \outp} + u_{\spec \pool\outp} &\leq 0, \quad && \fa \spec \in \specs,\ \outp \in \outputs  \label{eq:genbase1} \\
z_{\pool \outp} + x_{\pool \outp} & \leq C_{\outp}, && \fa \outp \in \outputs \label{eq:genbase2} .
\end{align}
\end{subequations}

A key property of these new decision variables is the relationship between
the flow and quality in the pool with the excess of the attribute
contributed by the flow through the pool
\begin{align}
\label{eq:oneqc}
u_{\spec \pool\outp} & = \xvar_{\pool\outp}t_{\spec \pool} && \fa \pool \in \pools,\ \outp \in \outputs,
\end{align}
which is valid because using \eqref{eq:pq_wdef} and \eqref{eq:pq_pq-1}
\begin{align*}
u_{\spec \pool\outp} = \sum_{\inp \in \inputs[\pool]}  \gamma_{\spec\inp\outp} w_{\inp\pool\outp} = \sum_{\inp \in
\inputs[\pool]}
\gamma_{\spec\inp\outp} q_{\inp\pool} x_{\pool\outp}
= \sum_{\inp \in \inputs[\pool]}
\gamma_{\spec\inp\outp} q_{\inp\pool} \sum_{i' \in \inputs[\pool]} w_{i'\pool\outp} =  t_{\spec \pool}\xvar_{\pool\outp}.
\end{align*}

In order to derive bounds on the new decision variables we define the parameters  $\ugamma_{\spec \pool}$ and $\ogamma_{\spec \pool}$ representing bounds on the excess of attribute $\spec$ over inputs that are connected to pool $\pool$, and $\uGamma_{\spec\pool\outp}$ and $\oGamma_{\spec\pool\outp}$ representing  bounds on the
excess of attribute $\spec$ over inputs that are connected to output $\outp$ via a by-pass flow :
\begin{align*}
 \ugamma_{\spec\pool} &= \min_{\inp \in \inputs[\pool]} \gamma_{\spec\inp} & \uGamma_{\spec\pool\outp} &= \min\bigl\{ \gamma_{\spec\inp} : i \in I_\outp \cup \bigcup_{\pool' \in \pools\setminus \{\pool\}}
  I_{\pool'} \bigr\}\\
 \ogamma_{\spec\pool} &= \max_{\inp \in \inputs[\pool]}
  \gamma_{\spec\inp}  & \oGamma_{\spec\pool\outp} &=  \max\bigl\{ \gamma_{\spec\inp} : i \in I_\outp \cup \bigcup_{\pool' \in \pools\setminus \{\pool\}} I_{\pool'} \bigr\}.
\end{align*}
We thus have,
\begin{subequations}
\label{eq:bndorig}
\begin{align}
   t_{\spec \pool} & \in [\ugamma_{\spec\pool},\ogamma_{\spec\pool}]  && \fa \spec \in \specs,\ \pool \in \pools \label{eq:bndorig2}\\
\uGamma_{\spec\pool\outp} z_{\pool \outp} & \leq y_{\spec \pool \outp}
                                            \leq
                                            \oGamma_{\spec\pool\outp}
                                            z_{\pool \outp},   \quad
                                                                      && \fa \spec \in \specs,\ \pool \in \pools,\ \outp \in \outputs.  \label{eq:bndorig1}
\end{align}
\end{subequations}

Despite the many relaxations performed in deriving this set, the \nonconvex relation \eqref{eq:oneqc}, which
relates the contribution of the excess from the pool to the attribute quality of the pool and the quantity passing
through the pool,
still preserves a key \nonconvex structure of the problem.

With these variables and constraints we now formulate the relaxation
of the pooling problem that we study. To simplify notation, we
drop the fixed indices $\pool,\outp,$ and $\spec$. Gathering the
constraints \eqref{eq:newvarcons}, \eqref{eq:oneqc}, and
\eqref{eq:bndorig}, together with nonnegativity on the $z$ and $\xvar$
variable, we define the set $T$ as those
$(\xvar,u,y,z,t) \in \mathbb{R}^5$ that satisfy:
\begin{align}
u  &= \xvar t \label{base0} \\
y + u &\leq 0 \label{eq:yu} \\
z + \xvar &\leq C \label{eq:zs} \\
y & \leq \oGamma z \label{eq:yz1} \\
y & \geq \uGamma z \label{eq:yz2} \\
z \geq 0, \
\xvar \in [0,C], \
t & \in [\ugamma,\ogamma].\notag
\end{align}
We can assume, without loss of generality, that $C=1$ by
scaling the variables $\xvar$, $u$, $y$, and $z$ by
$C^{-1}$.

Due to the nonlinear equation $u = \xvar t$, $T$ is a \nonconvex set
unless $\xvar$ or $t$ is fixed. Using the bounds $0 \leq \xvar \leq 1$
and $\ugamma \leq t \leq \ogamma$, the constraint $u=\xvar t$ can be
relaxed by the McCormick inequalities \cite{mccormick:76}:
\begin{align}
  u - \ugamma \xvar &\geq 0 \label{eq:us1} \\
  \ogamma \xvar - u &\geq 0 \label{eq:us2} \\
  u - \ugamma \xvar &\leq t - \ugamma \label{eq:ust1} \\
  \ogamma \xvar  - u &\leq \ogamma -t \label{eq:ust2} .
\end{align}
\Cref{eq:us1,eq:us2,eq:ust1,eq:ust2} provide the best possible convex
relaxation of the feasible points of $u = \xvar t$ given that $\xvar$ and $t$
are in the bounds mentioned above. However, replacing the \nonconvex constraint $u=\xvar t$ with these inequalities is not
sufficient to define $\conv\Of{T}$.

Note that \cref{eq:us1,eq:us2,eq:ust1,eq:ust2} imply the bounds $0 \leq \xvar
\leq 1$ and $\ugamma \leq t \leq \ogamma$. Also the bound constraint
$z\geq 0$ is implied by \cref{eq:yz1,eq:yz2}.
Thus, we define the
standard relaxation of the set $T$ by
\[ R^0 := \{ (\xvar,u,y,z,t) : \cref{eq:yu,eq:zs,eq:yz1,eq:yz2,eq:us1,eq:us2,eq:ust1,eq:ust2} \} . \]

Every convex set is described completely by its extreme points and
rays. The set $T$ is bounded and so has no extreme rays.
In \cite{luedtke_dambrosio_linderoth_schweiger:pooling_extreme_points}, we have characterized the extreme points of $T$, showing they are not finite. Thus, the convex hull of
$T$ is not a polyhedron.

\subsection{Valid Inequalities}
We now present the new valid inequalities for $\conv(T)$, two of them linear, and two of them convex
nonlinear. Depending on the signs of
$\ugamma$ and $\ogamma$, some of these inequalities are redundant. In the following, an inequality is said to be {\it
valid} for a set if every point in the set satisfies the inequality.

\begin{theorem}
  \label{thm:ineq1}
  If $\uGamma < 0$, then the following inequality is valid for $T$:
  \begin{equation}
    \label{eq:ineq1}
    (u - \uGamma \xvar)(u - \ugamma \xvar) \leq - \uGamma \xvar (t - \ugamma)
  \end{equation}
\end{theorem}

\begin{proof}
  Aggregating the inequalities \cref{eq:yu} (with weight 1), \cref{eq:zs} (with weight $-\uGamma$), and \cref{eq:yz2}
  (with weight 1) yields the inequality $u - \uGamma \xvar \leq -\uGamma$, which is valid for $R^0$.
  Multiplying this inequality by $\xvar (t - \ugamma) \geq 0$ on both sides yields the nonlinear inequality
  \begin{equation*}
    (u - \uGamma \xvar) \xvar(t - \ugamma) \leq (-\uGamma)\xvar(t-\ugamma)
  \end{equation*}
  which is also valid for $R^0$. Substituting $u=\xvar t$ into the left-hand-side of this yields \cref{eq:ineq1}.
\end{proof}

We observe that if $\ogamma < 0$, then \cref{eq:ineq1} is redundant.
Indeed, $\ogamma < 0$ implies $t < 0$ and therefore $u < 0$, which in
turn implies $u - \uGamma \xvar < -\uGamma \xvar$. On the other hand, $\xvar \leq
1$ and $t - \ugamma > 0$ imply that $t - \ugamma \geq \xvar t - \ugamma \xvar =
u - \ugamma \xvar$. Furthermore, $0 = u - \xvar t \leq u - \ugamma \xvar$ and
$-\uGamma \xvar \geq 0$.
Thus, we conclude that \cref{eq:ineq1} is always strict
if $\ogamma < 0$:
\begin{equation*}
  (u - \uGamma \xvar)(u - \ugamma \xvar)  < -\uGamma \xvar (u - \ugamma \xvar)
   \leq -\uGamma \xvar (t - \ugamma) .
\end{equation*}

We next show that \cref{eq:ineq1} is second-order cone representable
and thus convex. We can rewrite \cref{eq:ineq1} as:
\begin{align}
   (u - \uGamma \xvar)(u - \ugamma \xvar) \leq - \uGamma \xvar (t - \ugamma) \nonumber
  \ \Leftrightarrow \ & (u - \ugamma \xvar)^2 + (\ugamma - \uGamma)\xvar(u - \ugamma \xvar) \leq -
  \uGamma \xvar (t - \ugamma) \nonumber \\
  \Leftrightarrow \ & (u - \ugamma \xvar)^2 \leq \xvar \Bigl[ (-\uGamma)(t - \ugamma) + (\uGamma -
  \ugamma)(u - \ugamma \xvar) \Bigr] . \nonumber
\end{align}
This inequality has the form of a rotated second-order cone, $2 x_1x_2 \geq x_3^2$,
where $x_1  = \xvar/2$, $x_2  = (-\uGamma)(t - \ugamma) + (\uGamma - \ugamma)(u - \ugamma \xvar)$, and
$x_3 = u - \ugamma \xvar$.
Clearly, $x_1 \geq 0$.  The following lemma shows the \nonnegativity of $x_2$
and therefore establishes the second-order cone representability of \cref{eq:ineq1}.
\begin{lemma}
  If $\uGamma < 0$, the following inequality is valid for $T$:
  \begin{equation}
    \label{termpos}
    (-\uGamma)(t - \ugamma) + (\uGamma - \ugamma)(u - \ugamma \xvar) \geq 0
  \end{equation}
\end{lemma}
\begin{proof}
  First, as $\uGamma < 0$ then by $y+u \leq 0$, $-\uGamma(\xvar + z) \leq -\uGamma $, $\uGamma z - y \leq 0$, and $\ugamma \xvar -
  u \leq 0$, we have
  $ (\ugamma - \uGamma)\xvar  \leq -\uGamma$,
  and therefore, using $t - \ugamma \geq 0$,
  \begin{equation}
    \label{intermed}
    (\ugamma - \uGamma)(t - \ugamma) \xvar \leq (-\uGamma)(t - \ugamma).
  \end{equation}
  But then, using $u = t\xvar$, yields
  \[ (\ugamma - \uGamma)(t - \ugamma) \xvar = (\ugamma - \uGamma)(t\xvar -
  \ugamma \xvar)   = (\ugamma - \uGamma)(u -
  \ugamma \xvar). \]
  Substituting into \cref{intermed} and rearranging yields the result.
\end{proof}

The second inequality we derive is valid for points in $T$ with $y > 0$.
\begin{theorem}
  \label{thm:yikes}
  If $\oGamma > 0$ and $\ugamma < 0$, then the following inequality is valid for $T$ when $y > 0$:
  \begin{equation}
    \label{yikes}
    (\ogamma - \ugamma)y + \oGamma(\ogamma \xvar - u) + \frac{\ugamma y(u - \ugamma
      \xvar)}{y + u - \ugamma \xvar} \leq \oGamma ( \ogamma  - t )
  \end{equation}
\end{theorem}
\begin{proof}
  First, adding \cref{eq:zs} scaled by weight $\oGamma > 0$ to \cref{eq:yz1} yields the inequality
  \begin{equation}
    \label{basic}
    y + \oGamma \xvar \leq \oGamma
  \end{equation}
  which is valid for $R^0$.

  Next, using the substitution $u = \xvar t$ the left-hand-side of \eqref{yikes} becomes:
  \begin{align*}
    (\ogamma - \ugamma)y  & + \oGamma(\ogamma \xvar - u) + \frac{\ugamma y(u - \ugamma
      \xvar)}{y + u - \ugamma \xvar}  \\
    & = (\ogamma - \ugamma)y  + \oGamma \xvar (\ogamma  - t) + \frac{\ugamma y \xvar(t - \ugamma
      )}{y + \xvar(t - \ugamma)}  \\
    & \leq   (\ogamma - \ugamma)y +  \oGamma \xvar(\ogamma - t)   + \frac{\ugamma y\xvar (t - \ugamma)}{-\xvar t + \xvar(t - \ugamma)}  &&
    (\mbox{$y \leq -\xvar t$ and $\ugamma y\xvar(t-\ugamma) \leq 0$}) \\
    & =  (\ogamma - \ugamma)y + \oGamma \xvar(\ogamma - t)    - y(t - \ugamma) \\
    & = \oGamma \xvar (\ogamma - t) + y(\ogamma - t) \\
    & = (\oGamma \xvar + y) (\ogamma - t)  \leq \oGamma  (\ogamma - t)  && \mbox{ because $\ogamma \geq t$ and by \cref{basic}}.
  \end{align*}
\end{proof}
The
conditional constraint \eqref{yikes} cannot be directly
used in an algebraic formulation. We thus derive a convex reformulation for \eqref{yikes} that is valid also for $y \leq 0$. To this end define the
function $h: \reals\times\realsPos \rightarrow \reals$ as
\begin{align*}
  h\Of{y,v} \eqdef
  \begin{cases}
    0 & \text{if } y \leq 0\\
    (\ogamma - \ugamma)y + \ugamma g\Of{y,v} & \text{if } y > 0.
  \end{cases}
\end{align*}
where
\begin{align*}
  g\Of{y,v} \eqdef \frac{yv}{y + v} .
\end{align*}
We next show that $h$ is convex, and with $v=u - \ugamma \xvar$, can be used to define a constraint equivalent to \eqref{yikes}.
\begin{lemma}
  \label{thm:convex_extension}
  If $\oGamma > 0$ and $\ugamma < 0$, then the inequality
  \begin{align}
    \label{yikes_ext}
    \oGamma(\ogamma \xvar - u) + h\Of{y,u - \ugamma
  \xvar} \leq \oGamma ( \ogamma  - t )
  \end{align}
  is valid for $T$, $h$ is convex over $\reals\times\realsPos$, and any point $(\xvar, u, y, t)$ with $y > 0$ satisfies \eqref{yikes_ext}
  if and only if it satisfies \eqref{yikes}.
\end{lemma}

\begin{proof}
The statement that any point with $y > 0$ satisfies \eqref{yikes_ext}
  if and only if it satisfies \eqref{yikes} is immediate from the definition of $h$.

By \cref{thm:yikes}, inequality \eqref{yikes_ext} is satisfied by all points in $T$
  with $y > 0$. If $y \leq 0$ the inequality is also valid since
  \begin{align*}
    \oGamma(\ogamma \xvar - u) + h\Of{y,u,\xvar}
     = \oGamma(\ogamma \xvar - u)
     = \xvar \oGamma (\ogamma - t)
     \leq \oGamma (\ogamma - t).
  \end{align*}

We next show that $g$ is concave over $\realsPosStrict\times\realsPos$.
We use the python
library Sympy~\cite{sympy} to compute the Hessian of $g$, and obtain
\begin{align*}
    \left(\begin{matrix}\frac{2 v y}{\left(v + y\right)^{3}} - \frac{2 y}{\left(v + y\right)^{2}} & \frac{2 v y}{\left(v + y\right)^{3}} - \frac{v}{\left(v + y\right)^{2}} - \frac{y}{\left(v + y\right)^{2}} + \frac{1}{v + y}\\\frac{2 v y}{\left(v + y\right)^{3}} - \frac{v}{\left(v + y\right)^{2}} - \frac{y}{\left(v + y\right)^{2}} + \frac{1}{v + y} & \frac{2 v y}{\left(v + y\right)^{3}} - \frac{2 v}{\left(v + y\right)^{2}}\end{matrix}\right)
\end{align*}
and its Eigenvalues as
$\lambda_1  = 0$ and
$\lambda_2  = - (2 v^{2} + 2 y^{2})/(v + y)^{3}$.
The second Eigenvalue $\lambda_2$ is negative because $v \geq 0$, $y > 0$. The
Hessian is therefore negative semidefinite and $g$ is concave.


Finally, we show that $h$
  is convex. Let $p_i = (y_i,v_i) \in \reals\times\realsPos$ for $i=1,2$ and
  $\lambda \in (0,1)$. We need to show that
  \begin{align}
    \label{thm:yikes_conv}
    h\Of{\lambda p_1 + (1-\lambda) p_2} \leq \lambda h\Of{p_1} +
    (1-\lambda) h\Of{p_2}.
  \end{align}
  If $y_i > 0$, $i=1,2$, then \eqref{thm:yikes_conv} holds because $g$ is concave over $\realsPosStrict\times\realsPos$,
  and hence because $\ugamma < 0$, $h$ is convex over this region.
  If $y_i \leq 0$, $i=1,2$, then then \eqref{thm:yikes_conv} holds because $h$ is linear over the points with $y \leq 0$.
  Therefore, assume without loss of generality $y_1 \leq 0$ and
  $y_2 > 0$. First we show that $h$ is \nonnegative. For $y \leq 0$
  this is clear and for $y > 0$ we have
  \begin{align*}
    h\Of{u,v}  = (\ogamma - \ugamma)y + \frac{\ugamma yv}{y + v}
               = \frac{(\ogamma - \ugamma)y (y + v) + \ugamma yv}{y
                + v}
               = \frac{(\ogamma - \ugamma)y^2 + \ogamma yv}{y + v}
                \geq 0.
  \end{align*}
  Furthermore we know $h\Of{p_1} = 0$ since $y_1 \leq 0$. If
  $h\Of{\lambda p_1 + (1-\lambda) p_2} = 0$, then
  \eqref{thm:yikes_conv} is always fulfilled. If
  $h\Of{\lambda p_1 + (1-\lambda) p_2}$ does not vanish, then denote
  by $p_3 = (y_3, v_3)$ the point on the line between $p_1$ and $p_2$
  with $y_3 = 0$. If $v_3 = 0$, then $v_2 = 0$ since $v_1 \geq 0$. In
  this case $g$ vanishes such that $h$ is linear 
  between $p_3$ and $p_2$ and \eqref{thm:yikes_conv} is fulfilled. If
  $v_3 > 0$, then $g$ is also well-defined at $p_3$ with
  $g\Of{p_3} = 0$ so that
  $h(p_3) = (\ogamma - \ugamma)y_3 + \ugamma g\Of{p_3} = 0$ and $h$ is convex
  on the line between $p_3$ and $p_2$. Furthermore, there exists a
  $\hat{\lambda}$ such that
   \begin{align*}
     \lambda p_1 + (1-\lambda) p_2 = \hat{\lambda} p_3 + (1-\hat{\lambda}) p_2
   \end{align*}
   and since $p_3$ is closer to $p_2$ than $p_1$, it holds that $\lambda \leq
   \hat{\lambda}$.
   Finally, we can show that \eqref{thm:yikes_conv} also holds in this
   case:
   \begin{align*}
     h\Of{\lambda p_1 + (1-\lambda) p_2}
       & = h\Of{\hat{\lambda} p_3 + (1-\hat{\lambda}) p_2} \\
       & \leq \hat{\lambda} h\Of{p_3} + (1-\hat{\lambda}) h\Of{p_2} \\
       & = \lambda h\Of{p_1} + (1-\hat{\lambda}) h\Of{p_2}
        \leq\lambda h\Of{p_1} + (1-\lambda) h\Of{p_2} .
   \end{align*}
\end{proof}

The remaining two new valid inequalities we present are linear.
\begin{theorem}
  If $\oGamma > 0$, then the following inequality is valid for $T$:
  \begin{align}
    \label{eq:beast_case2}
    (\ogamma-\ugamma)y + \ugamma (\ogamma \xvar - u) + \oGamma (u - \ugamma \xvar) \leq \oGamma(t - \ugamma) .
  \end{align}
\end{theorem}
\begin{proof}
  First observe that $y + u \leq 0$ and $-u \leq -\ugamma \xvar$ together imply
  \begin{equation}
    \label{ysimp}
    y \ \leq -\xvar\ugamma.
  \end{equation}
  Next,
  \begin{align*}
    (\ogamma - \ugamma)y &=   (t - \ugamma)y + (\ogamma - t)y \\
    & \leq (t - \ugamma)\oGamma z +(\ogamma - t)y  &\quad& \text{because $y \leq \oGamma z$ and $t - \ugamma \geq 0$} \\
    & \leq (t - \ugamma)\oGamma z +(\ogamma - t)(-\xvar\ugamma) & & \text{by \cref{ysimp} and $\ogamma - t \geq 0$}
  \end{align*}
  and thus
  \begin{equation}
    \label{case2proof1}
    (\ogamma - \ugamma)y - (t - \ugamma)\oGamma z + (\ogamma - t)(\xvar\ugamma)  \leq 0 .
  \end{equation}
  Then, multiply the inequality $z+\xvar\leq 1$ on both sides by $\oGamma(t - \ugamma) \geq 0$ to yield:
  \begin{equation}
    \label{case2proof2}
    (t - \ugamma)\oGamma z + (t - \ugamma)\oGamma \xvar \leq (t - \ugamma)\oGamma
  \end{equation}
  Adding \cref{case2proof1,case2proof2} yields:
  \begin{equation}
    \label{case2nlform}
    (\ogamma - \ugamma)y + (\ogamma - t)\xvar \ugamma + (t - \ugamma)\oGamma \xvar \leq  (t - \ugamma)\oGamma
  \end{equation}
  Finally, substituting $u = \xvar t$ from \cref{base0} yields \cref{eq:beast_case2}.
\end{proof}

We next show that if $\ogamma > 0$, then \cref{eq:beast_case2} is redundant.
Assuming  $\ogamma > 0$, then scaling the inequality $-u + \ugamma \xvar \leq 0$ by
$\ogamma > 0$ and combining that with the valid inequality $(\ogamma - \ugamma) y + (\ogamma - \ugamma) u \leq 0$ and yields
\begin{equation}
  \label{c2:red}
  (\ogamma - \ugamma) y - \ugamma u + \ugamma \ogamma \xvar \leq 0.
\end{equation}
But, also since $u - \ugamma \xvar \leq t - \ugamma$, it follows that
$ \oGamma ( u - \ugamma \xvar) \leq \oGamma (t - \ugamma)$ .  Combining this with \cref{c2:red} implies \cref{eq:beast_case2}.

The next theorem presents the last valid inequality for $T$ in this section.
\begin{theorem}
  \label{lem:beast_case3}
  If $\uGamma < 0$, then the following inequality is valid for $T$:
  \begin{align}
    \label{eq:beast_case3}
    (\ugamma - \uGamma)(\ogamma \xvar - u) \leq -\uGamma (\ogamma - t) .
  \end{align}
\end{theorem}
\begin{proof}
  Aggregate inequality \cref{eq:zs} with weight $-\uGamma > 0$ yields
  \begin{align}
    \label{eq:beast_case3_base2_weighted}
    -\uGamma (z + \xvar) \leq -\uGamma.
  \end{align}
  Furthermore, using $y \geq \uGamma z$, $\uGamma < 0$, and
  \cref{ysimp}, yields
  $ -\uGamma z +  \xvar \ugamma \leq 0$,
  which combined with
  \cref{eq:beast_case3_base2_weighted}
  yields
  \begin{align*}
    (\ugamma-\uGamma)\xvar \leq -\uGamma .
  \end{align*}
  Multiplying both sides of this inequality by $\ogamma - t \geq 0$  yields
  \[ (\ugamma-\uGamma)\xvar(\ogamma - t) \leq -\uGamma (\ogamma - t) . \]
  Substituting $\xvar t = u$ yields \cref{eq:beast_case3}.
\end{proof}

If $\ugamma < 0$, then $(\ugamma - \uGamma)(\ogamma \xvar - u) \leq -\uGamma(\ogamma \xvar - 0) \leq -\uGamma (\ogamma - t)$ and
so \cref{eq:beast_case3} is redundant.


\section{Convex hull analysis}\label{sec:convex_hull_analysis}
We next demonstrate that the set $R^0$ combined with certain subsets
of the new valid inequalities, depending on the sign of $\ugamma$ and
$\ogamma$, are sufficient to define the convex hull of $T$. Let us
first restate the relevant inequalities for convenience:
\begin{align}
  (u - \uGamma \xvar)(u - \ugamma \xvar) &\leq - \uGamma \xvar (t - \ugamma) \tag{\ref{eq:ineq1}} \label{eq:nl1} \\
  (\ogamma - \ugamma)y + \oGamma(\ogamma \xvar - u) + \frac{\ugamma y (u -
    \ugamma
    \xvar)}{y + u - \ugamma \xvar} & \leq \oGamma (\ogamma  - t) \quad \mathrm{if} \ y > 0 . \tag{\ref{yikes}} \label{eq:nl2}\\
  (\ogamma-\ugamma)y + \ugamma (\ogamma \xvar - u) + \oGamma (u - \ugamma \xvar) &\leq \oGamma(t - \ugamma)
  \tag{\ref{eq:beast_case2}} \label{c2:lin}\\
  (\ugamma - \uGamma)(\ogamma \xvar - u) &\leq -\uGamma (\ogamma - t) .
  \tag{\ref{eq:beast_case3}} \label{c3:lin}
\end{align}
Next, we define the sets which include the nonredundant valid inequalities for
different signs of $\ugamma$ and
$\ogamma$:
\begin{align*}
  R^1 & = \{ (\xvar,u,y,z,t) \in R^0 : \text{\cref{eq:nl1,eq:nl2}} \}, \\
  R^2 & = \{ (\xvar,u,y,z,t) \in R^0 : \text{\cref{eq:nl2,c2:lin}} \}, \\
  R^3 & = \{ (\xvar,u,y,z,t) \in R^0 : \text{\cref{eq:nl1,c3:lin}} \} .
\end{align*}
The following theorems show that $R^1, R^2$, and $R^3$ describe the
convex hull of $T$ in different cases. Since all inequalities that
define the sets are valid for $T$ and convex, $R^i$ are convex
relaxations for $T$ and we know $\conv(T) \subseteq R^i$ for
$i=1,2,3$. To show that a relaxation defines $\conv(T)$ in a
particular case, we show that every extreme point of the relaxation
satisfies the \nonconvex constraint $u=\xvar t$ even though this
equation is not enforced in the relaxation. The three main theorems are stated next,
and are proved using lemmas that are stated and proved
thereafter. \Cref{tab:pooling:pagenumbers} gives an overview of which of
the lemmas is used in the proof of each theorem and on which page the
lemmas are found.

\newcommand{\pageline}[2]{\cref{#1} & \cref{#2} & Page \pageref{#1}\\}

\begin{table}[t]
  \setlength{\tabcolsep}{20pt}
  \centering
  \begin{tabular}{lll}
    \toprule
    Result & Used in proof of & Stated on\\
    \midrule
    \pageline{easycases}{thm:case1,thm:case2,thm:case3}
    \pageline{uless}{thm:case1}
    \pageline{umore}{thm:case1}
    \pageline{c2:uless}{thm:case2}
    \pageline{c2:umore}{thm:case2}
    \pageline{c3:uless}{thm:case3}
    \pageline{c3:umore}{thm:case3}
                       \bottomrule
  \end{tabular}
  \caption{Overview over lemmas used in the proofs of
    \cref{thm:case1,thm:case2,thm:case3}}
  \label{tab:pooling:pagenumbers}
\end{table}

\begin{theorem}
  \label{thm:case1}
  Assume $\ugamma < 0 < \ogamma$ and $\uGamma < 0 < \oGamma$. Then  $\conv(T) = R^1$.
\end{theorem}
\begin{proof}
  Being a bounded convex set, $R^1$ is completely characterized by its
  extreme points. We prove that every extreme point of $R^1$ is
  in $T$, \ie fulfills the equation $u=\xvar t$. Every point in $R^1$ can
  thus be represented as a convex combination of points in $T$ and
  $R^1 = \conv\Of{T}$ is proved.

  It remains to show that $u=\xvar t$ for all $p = (\xvar,u,y,z,t) \in
  \ext\Of{R^1}$. \Cref{easycases} shows that this is the case for $\xvar$ or
  $t$ are at its bounds, \ie if $\xvar\in\zeroOne$ or $t\in\{\ugamma,
  \ogamma\}$. Under the condition $0 < \xvar < 1$ and $\ugamma < t <
  \ogamma$,
  \cref{uless,umore} show that points with $u < \xvar t$ and $u > \xvar t$,
  respectively, cannot be extreme. 
\end{proof}

\begin{theorem}
  \label{thm:case2}
  Assume $\ugamma < \ogamma < 0$ and $\uGamma < 0 < \oGamma$. Then  $\conv(T) = R^2$.
\end{theorem}
\begin{proof}
  Based on the same argument as in the proof of \cref{thm:case1}, we
  show that that $u=\xvar t$ for all $p = (\xvar,u,y,z,t) \in \ext\Of{R^2}$.
  \Cref{easycases} shows that for $\xvar\in\zeroOne$ or
  $t\in\{\ugamma, \ogamma\}$. For $0 < \xvar < 1$ and
  $\ugamma < t < \ogamma$ and under the assumptions of this theorem
  \cref{c2:uless,c2:umore} show that points with $u < \xvar t$ and
  $u > \xvar t$, respectively, cannot be extreme.
\end{proof}

\begin{theorem}
  \label{thm:case3}
  Assume $0 < \ugamma < \ogamma$ and $\uGamma < 0 < \oGamma$. Then  $\conv(T) = R^3$.
\end{theorem}
\begin{proof}
  The proof is analogous to the proofs of \cref{thm:case1,thm:case2},
  but in this case \cref{c3:uless,c3:umore} show that
  points with $u < \xvar t$ and $u > \xvar t$, respectively, cannot be extreme.
\end{proof}

\subsection{Preliminary Results}
\label{sec:proofs}

In the following, for different assumptions on the sign of
$\ugamma$ and $\ogamma$, we demonstrate that if $p=(\xvar,u,y,z,t)$ has
either $u > \xvar t$ or $u < \xvar t$, then $p$ is not an extreme point of $\conv\Of{T}$. This is accomplished by considering different cases
and in each case, we provide two distinct points which depend on a
parameter $\epsilon > 0$, denoted $p_i^{\epsilon}$, $i=1,2$, which
satisfy $p = (1/2)p_1^{\epsilon} + (1/2)p_2^{\epsilon}$. Furthermore,
the points $p_i^{\epsilon}$ are defined such that $p_i^{\epsilon}
\rightarrow p$ as $\epsilon \rightarrow 0$. The points are then shown
to be in the given relaxation for $\epsilon > 0$ small enough,
providing a proof that $p$ is not an extreme point of the relaxation.
To show the points are in a given relaxation for $\epsilon > 0$ small
enough, for each inequality defining the relaxation we either directly
show the points satisfy the inequality, or else we show that the point
$p$ satisfies the inequality with strict inequality. In the latter
case, the following lemma ensures that both points $p_1^{\epsilon}$
and $p_2^{\epsilon}$ satisfy the constraint if $\epsilon$ is small
enough.
\begin{lemma}
  \label{lem:convergence}
  Let $p^{\epsilon}: \reals_+ \rightarrow \reals^n$  with
  $  \lim_{\epsilon\rightarrow 0} p^{\epsilon} = p$
  for some $p\in\reals^n$. Suppose $ap < b$ for $a \in \reals^n$, $b
  \in \reals$. Then there exists an
  $\hat{\epsilon} > 0$ such that
  \begin{equation*}
    ap^\epsilon < b \qquad \fa \epsilon < \hat{\epsilon}.
  \end{equation*}
\end{lemma}
\begin{proof}
  Follows directly from continuity of the function $f\Of{x} = ax-b$.
\end{proof}

Throughout this section, for $\epsilon > 0$, we use the notation:
\begin{align*}
  \alphep_1 = 1 - \epsilon, \
  \alphep_2 = 1 + \epsilon & \qquad \text{and} \qquad
  \delta_1^{\epsilon}  = \epsilon, \
  \delta_2^{\epsilon}  = - \epsilon.
\end{align*}
Obviously,
$  \lim_{\epsilon \rightarrow 0} \alphep_i  = 1$  and
$  \lim_{\epsilon \rightarrow 0} \delta_i^{\epsilon} = 0$
for $i\in\{1,2\}$.

The series of Lemmas that prove \cref{thm:case1,thm:case2,thm:case3}
is started by \cref{easycases} which applies to all cases and tells us
that points on the boundaries of the domains of $\xvar$ and $t$ fulfill $u=\xvar t$.
\begin{lemma}
  \label{easycases}
  Let $p = (\xvar,u,y,z,t) \in R^0$. If $\xvar =0$, $\xvar =1$, $t = \ugamma$, or $t=\ogamma$, then $u=\xvar t$.
\end{lemma}
\begin{proof}
  This follows since \cref{eq:us1,eq:us2,eq:ust1,eq:ust2} are the McCormick inequalities for relaxing the constraint $u=\xvar t$ over
  $\xvar \in[0,1]$ and $t \in [\ugamma,\ogamma]$, and it is known (\eg \cite{alkhayyal.falk:83}) that if either of the variables are at its
  bound, then the McCormick inequalities ensure that $u=\xvar t$.
\end{proof}
As \cref{easycases} applies to all the cases we assume from now on
that $0 < \xvar < 1$ and $\ugamma < t < \ogamma$. We use the following two propositions in several places in this section.
\begin{proposition}
  \label{ustslack}
  Suppose $\uGamma < 0$. Let $p = (\xvar,u,y,z,t) \in R^0$ with $0 < \xvar < 1$ and $\ugamma < t < \ogamma$.
  \begin{enumerate}
  \item If $u < \xvar t$, then $p$ satisfies \cref{eq:us2,eq:ust1,eq:nl1} with strict inequality.
  \item If $u > \xvar t$, then $p$ satisfies \cref{eq:us1,eq:ust2} with strict inequality.
  \end{enumerate}
\end{proposition}
\begin{proof}
  1. Suppose $u < \xvar t$. Then, $\ogamma \xvar - u > \ogamma \xvar - \xvar t = \xvar (\ogamma - t) > 0$, and so $p$ satisfies \cref{eq:us2}
  with strict inequality.
  Next,
  \begin{equation}
    \label{tmpstrict}
    u - \ugamma \xvar < \xvar t - \ugamma \xvar = \xvar(t - \ugamma) < t - \ugamma
  \end{equation}
  as $\xvar < 1$ and $t > \ugamma$, and so \cref{eq:ust1} is satisfied by
  $p$ with strict inequality.
  To show that \cref{eq:nl1} is satisfied strictly, we aggregate \cref{eq:yu} with weight $1$, \cref{eq:zs} with weight $-\uGamma$, and \cref{eq:yz2} with
  weight $1$ and get
  \begin{align}
    u - \uGamma \xvar &\leq -\uGamma. \label{pv1}
  \end{align}
  As $u - \ugamma \xvar \geq 0$,
  \[ (u - \uGamma \xvar)(u - \ugamma \xvar) \leq -\uGamma (u - \ugamma \xvar) < -\uGamma \xvar(t - \ugamma) \]
  where the last inequality follows from \cref{tmpstrict} and $\uGamma < 0$,  and thus \cref{eq:nl1} is satisfied by $p$
  with strict inequality.

  2. Now suppose $u > \xvar t$.  Then,
  $u - \xvar\ugamma > \xvar t - \xvar\ugamma = \xvar(t-\ugamma) > 0$ and so $p$ satisfies \cref{eq:us1} with strict inequality. Next,
  \begin{equation}
    \label{tmpstrict2}
    \ogamma \xvar - u < \ogamma \xvar  - \xvar t = \xvar(\ogamma  - t) < \ogamma - t
  \end{equation}
  as $\xvar < 1$ and $t < \ogamma$, and so \cref{eq:ust2} is satisfied with strict inequality.
\end{proof}
\begin{proposition}
  \label{c1:ustslack}
  Suppose $\uGamma < 0$ and $\ugamma < 0$. Let $p = (\xvar,u,y,z,t) \in
  R^0$  with $0 < \xvar < 1$ and $\ugamma < t < \ogamma$.
  If $u > \xvar t$ and  $y > 0$, then $p$ satisfies \cref{eq:nl2} with strict inequality.
\end{proposition}
\begin{proof}
  Then, as $\ugamma y < 0$, we have
  \begin{equation*}
    \ugamma y ( u - \ugamma \xvar) < \ugamma y (\xvar t - \ugamma \xvar) = \ugamma y \xvar (t - \ugamma) .
  \end{equation*}
  Thus, as $y + u - \ugamma \xvar > 0$,
  \begin{equation*}
    \frac{\ugamma y ( u - \ugamma \xvar)}{y + u - \ugamma \xvar} < \frac{\ugamma y \xvar (t - \ugamma) }{y + u - \ugamma \xvar}
    \leq \frac{\ugamma y \xvar (t - \ugamma) }{-\ugamma \xvar} = -y(t - \ugamma)
  \end{equation*}
  where the last inequality follows from $y+u \leq 0$ and $\ugamma y \xvar (t - \ugamma) < 0$.
  Thus,
  \begin{align*}
    (\ogamma - \ugamma)y + \oGamma(\ugamma \xvar - u) + \frac{\ugamma y (u - \ugamma
      \xvar)}{y + u - \ugamma \xvar} & < (\ogamma - \ugamma)y   + \oGamma(\ugamma \xvar - u) - y  (t - \ugamma) \\
    & < y (\ogamma - t) + \oGamma \xvar (\ogamma - t) \\
    & = (y + \oGamma \xvar)(\ogamma - t) \leq \oGamma (\ogamma - t)
  \end{align*}
  where the last inequality follows from $\ogamma - t > 0$ and the
  fact that aggregating \cref{eq:zs} with weight $\oGamma$ and
  \cref{eq:yz1} yields $y + \oGamma \xvar \leq \oGamma$.
\end{proof}

\subsection{Proof of \cref{thm:case1}}
\label{sec:prov_case1}

We now state and prove the two main lemmas that support the proof of \cref{thm:case1}.
\begin{lemma}
  \label{uless}
  Suppose $\uGamma < 0 < \oGamma$. Let $p = (\xvar,u,y,z,t) \in R^1$ with $0 < \xvar < 1$ and $\ugamma < t < \ogamma$. If $u < \xvar t$, then $p$ is not an extreme point of $R^1$.
\end{lemma}
\begin{proof}
  We consider four cases: (a) $y+u < 0$, (b) $z+\xvar < 1$, (c) $\uGamma z - y < 0$ and $y - \oGamma z < 0$, and (d)
  $z+\xvar=1$, $y+u=0$, and either $\uGamma z - y = 0$ or $y - \oGamma z = 0$.
  In each of them we define a series of points $p_i^{\epsilon} =
  (\xvar_i^{\epsilon}, u_i^{\epsilon}, y_i^\epsilon, z_i^\epsilon,
  t_i^{\epsilon})$ for $i \in \{1,2\}$ that depends on $\epsilon > 0$
  with $p = 0.5 ( p_1^{\epsilon} + p_2^{\epsilon} )$ and which
  satisfy  $\lim_{\epsilon \rightarrow 0} p_i^{\epsilon} = p$.
  We then show that both $p_i^{\epsilon}$ are in $R^1$ and thus
  $p$ is not an extreme point of $R^1$. To show $p_i^{\epsilon} \in
  R^1$, we need to ensure that it satisfies all inequalities defining
  $R^1$. For those inequalities that satisfied strictly at $p$,
  \cref{lem:convergence} ensures that this is the case. For the
  remaining inequalities, we show it directly.

  By \cref{ustslack}, $u < \xvar t$ implies that $p$ satisfies
  \cref{eq:us2,eq:ust1,eq:nl1} with strict
  inequality. 
  It remains to show that the points $p_i^{\epsilon}$ satisfy
  \cref{eq:yu,eq:zs,eq:yz1,eq:yz2,eq:us1,eq:ust2,eq:nl2} for
  $\epsilon > 0$ small enough. Note that $z \geq 0$ is implied by
  \cref{eq:yz1,eq:yz2} and does not have to be proved explicitly.

  \proofcase{a}{$y+u < 0$}
  For $\epsilon > 0$, define
  $p_i^{\epsilon} = (\xvar_i^{\epsilon}, u_i^{\epsilon}, y_i^\epsilon, z_i^\epsilon, t_i^{\epsilon})$
  where, for $i=1,2$,
  \begin{alignat*}{5}
    \xvar_i^{\epsilon} & \eqdef (1 - \alphep_i) + \alphep_i \xvar, \ \ &
    u_i^{\epsilon}   & \eqdef \ugamma (1 - \alphep_i) + \alphep_i u, \ \ &
    y_i^{\epsilon}   & \eqdef \alphep_i y, \ \ \\
    z_i^{\epsilon} & \eqdef \alphep_i z, \ \ &
    t_i^{\epsilon} & \eqdef (1-\alphep_i) \ugamma + \alphep_i t .
  \end{alignat*}
  Since $\alphep_i$ converge to 1, it is clear that
  $p_i^{\epsilon}$ converges to $p$ and \cref{lem:convergence}
  can be applied.
  In the following we check that $p_i^{\epsilon}$ satisfies the remaining inequalities.
  \begin{description}[nosep,leftmargin=2.5em]
  \item[\cref{eq:yu}] Satisfied strictly by $p$ by the assumption of
    this case.
  \item[\cref{eq:zs,eq:yz1,eq:yz2}] Easily checked directly.
  \item[\cref{eq:us1}] Follows from
    \[ u_i^{\epsilon} - \ugamma \xvar_i^{\epsilon} = \ugamma (1-\alphep_i) + \alphep_i u - \ugamma((1-\alphep_i) + \alphep_i \xvar)
    = \alphep_i( u - \ugamma \xvar) \geq 0 . \]
  \item [\cref{eq:ust2}] Follows from
    \begin{align*}
      \ogamma \xvar_i^{\epsilon} - u_i^{\epsilon} &= \ogamma ((1-\alphep_i) + \alphep_i \xvar) - \ugamma(1 - \alphep_i) - \alphep_i
      u \\
      &= (1 - \alphep_i)(\ogamma - \ugamma) + \alphep_i (\ogamma \xvar - u) \\
      &\leq (1 - \alphep_i)(\ogamma - \ugamma) + \alphep_i (\ogamma - t)
      = \ogamma - (1 - \alphep_i)\ugamma  - \alphep_i t  = \ogamma - t_i^{\epsilon} .
    \end{align*}
  \item[\cref{eq:nl2}] If $y > 0$, then also $y_{i}^{\epsilon} > 0$,
    and using $u_i^{\epsilon} - \ugamma \xvar_i^{\epsilon} =\alphep_i( u -
    \ugamma \xvar)$ and $\ogamma \xvar_i^{\epsilon} - u_i^{\epsilon} = (1 -
    \alphep_i)(\ogamma - \ugamma) + \alphep_i (\ogamma \xvar - u)$, the
    left-hand-side of \cref{eq:nl2} evaluated at $p_i^{\epsilon}$
    equals:
  \begin{align*}
    & \alphep_i \Bigl( (\ugamma - \ogamma) y + \oGamma (\ogamma \xvar - u)
    + \frac{ \ugamma y (u - \ugamma \xvar)}{y + u + \ugamma
      \xvar} \Bigr) + \oGamma(1 - \alphep_i)(\ugamma - \ogamma) \\
    & \leq \alphep_i \oGamma (\ogamma - t) +
    \oGamma(1-\alphep_i)(\ogamma - \ugamma) = \oGamma (\ogamma -
    \alphep_i t - (1 - \alphep_i) \ugamma) = \oGamma (\ogamma -
    t_i^{\epsilon})
  \end{align*}
  and hence \cref{eq:nl2} is satisfied by $p_i^\epsilon$ for $i=1,2$
  and any $\epsilon \in (0,1)$ when $y > 0$. On the other hand, if $y
  \leq 0$, then $y_{i}^{\epsilon} \leq 0$, and $p_i^{\epsilon}$ is not
  required to satisfy \cref{eq:nl2} for $i=1,2$.
  \end{description}

  \proofcase{b}{$z+\xvar<1$}
  For $\epsilon > 0$, define
  $p_i^{\epsilon} = (\xvar_i^{\epsilon}, u_i^{\epsilon}, y_i^\epsilon, z_i^\epsilon, t_i^{\epsilon})$
  where, for $i=1,2$,
  \begin{alignat*}{5}
    \xvar_i^{\epsilon} & \eqdef \alphep_i \xvar, \ \ &
    u_i^{\epsilon}     & \eqdef \alphep_i u, \ \ &
    y_i^{\epsilon}     & \eqdef \alphep_i y, \ \ &
    z_i^{\epsilon}     & \eqdef \alphep_i z, \ \ &
    t_i^{\epsilon}     & \eqdef  \alphep_i t + (1-\alphep_i)\ogamma .
  \end{alignat*}
  \begin{description}[nosep,leftmargin=2.5em]
  \item[\cref{eq:yu}] Easily checked directly.
  \item[\cref{eq:zs}] Satisfied strictly by $p$ by the assumption of
    this case.
  \item[\cref{eq:yz1,eq:yz2,eq:us1}] Easily checked directly.
  \item[\cref{eq:ust2}] Follows from
    \[ \ogamma \xvar_i^{\epsilon} - u_i^{\epsilon} = \alphep_i (\ogamma \xvar
    - u) \leq \alphep_i (\ogamma - t) = \alphep_i \ogamma - \alphep_i
    t = \alphep_i \ogamma - (1-\alphep_i)\ogamma - t_i^{\epsilon} =
    \ogamma - t_i^{\epsilon} .
    \]
  \item[\cref{eq:nl2}] If $y > 0$, then also $y_{i}^{\epsilon} > 0$,
    and the left-hand-side of \cref{eq:nl2} evaluated at
    $p_i^{\epsilon}$ equals:
    \begin{align*}
      \alphep_i \Bigl( (\ugamma - \ogamma) y + \oGamma (\ogamma \xvar - u)
      + \frac{ \ugamma y (u - \ugamma \xvar)}{y + u + \ugamma \xvar} \Bigr) &
      \leq
      \alphep_i \oGamma (\ogamma - t) \\
      & = \oGamma ( \alphep_i \ogamma -
      t_i^{\epsilon} + (1-\alphep_i)\ogamma)
       = \oGamma (\ogamma - t_i^{\epsilon})
    \end{align*}
    and hence \cref{eq:nl2} is satisfied by $p_i^\epsilon$ for $i=1,2$ and any $\epsilon \in (0,1)$ when $y > 0$. On the
    other hand, if $y \leq 0$, then $y_{i}^{\epsilon} \leq 0$, and $p_i^{\epsilon}$ is not required to satisfy
    \cref{eq:nl2} for $i=1,2$.
  \end{description}

  \proofcase{c}{$\uGamma z - y < 0$ and $y - \oGamma z < 0$}
  For $\epsilon > 0$, define
  $p_i^{\epsilon} = (\xvar_i^{\epsilon}, u_i^{\epsilon}, y_i^\epsilon, z_i^\epsilon, t_i^{\epsilon})$
  where
  \begin{alignat*}{5}
    \xvar_i^{\epsilon} & \eqdef \alphep_i \xvar, \ \ &
    u_i^{\epsilon} & \eqdef \alphep_i u, \ \ &
    y_i^{\epsilon} & \eqdef \alphep_i y, \ \ &
    z_i^{\epsilon} & \eqdef (1 - \alphep_i)+\alphep_i z, \ \ &
    t_i^{\epsilon} & \eqdef  \alphep_i t + (1-\alphep_i)\ogamma,
  \end{alignat*}
  for $i=1,2$.
  \begin{description}[nosep,leftmargin=2.5em]
  \item[\cref{eq:yu}] Easily checked directly.
  \item[\cref{eq:zs}] Follows from
    \[ z_i^{\epsilon} + \xvar_i^{\epsilon} = (1-\alphep_i) + \alphep_i z +
    \alphep_i \xvar = (1-\alphep_i) + \alphep_i (z + \xvar) \leq 1 . \]
  \item[\cref{eq:yz1}, \cref{eq:yz2}] Satisfied strictly by $p$ by the assumption of
    this case.
  \item[\cref{eq:us1}] Easily checked directly.
  \item[\cref{eq:ust2}, \cref{eq:nl2}] As the definitions of
    $t_i^{\epsilon}, y_i^{\epsilon}, \xvar_i^{\epsilon}$, and
    $u_i^{\epsilon}$ are the same as in Case (b), it follows from the
    arguments in that case that $p_i^{\epsilon}$ satisfies
    \cref{eq:ust2,eq:nl2} for $i=1,2$ and any $\epsilon
    \in (0,1)$.
  \end{description}

  \proofcase{d}{$z+\xvar=1$, $y+u=0$, and either $\uGamma z - y = 0$ or
    $y - \oGamma z = 0$}
  For $\epsilon > 0$, define $p_i^{\epsilon} = (\xvar_i^{\epsilon},
  u_i^{\epsilon}, y_i^\epsilon, z_i^\epsilon, t_i^{\epsilon})$ where, for $i=1,2$,
  \begin{alignat*}{5}
    \xvar^{\epsilon}_i & \eqdef \xvar - \delta^{\epsilon}_i, \ \  &
    u^{\epsilon}_i & \eqdef u - \uGamma \delta^{\epsilon}_i, \ \ &
    y^{\epsilon}_i & \eqdef y + \uGamma \delta^{\epsilon}_i, \ \ &
    z^{\epsilon}_i & \eqdef z + \delta^{\epsilon}_i, \ \ &
    t^{\epsilon}_i & \eqdef t + \delta^{\epsilon}_i (\ogamma - \uGamma) .
  \end{alignat*}

  \begin{description}[nosep,leftmargin=2.5em]
  \item[\cref{eq:yu}, \cref{eq:zs}] Easily checked directly.
  \item[\cref{eq:yz1}] We show that when $z+\xvar=1$ and $y+u=0$, then $y
    - \oGamma z < 0$. Indeed, if $y - \oGamma z = 0$, then as $z+\xvar=1$,
    it follows that
    \begin{equation}
      \label{yinterm}
      y + \oGamma \xvar = \oGamma
    \end{equation}
    Then, using $y = \oGamma z > 0$, and evaluating $p$ in the left-hand-side of \cref{eq:nl2} yields
    \begin{alignat*}{2}
      &(\ogamma - \ugamma) y + \oGamma (\ogamma \xvar - u) + \frac{\ugamma y (u - \ugamma \xvar)}{y + u - \ugamma \xvar}  \\
      &=  (\ogamma - \ugamma) y + \oGamma (\ogamma \xvar - u) + \frac{\ugamma y (u - \ugamma \xvar)}{- \ugamma \xvar}  &\quad&
      \mathrm{since } \ y+u = 0 \\
      &= \ogamma (y + \oGamma \xvar) - u( \oGamma  + y/\xvar) \\
      &> \ogamma (y + \oGamma \xvar) - \xvar t (\oGamma + y/\xvar) & & \mathrm{since }\  u < \xvar t \ \mathrm{ and } \  \oGamma + y/\xvar > 0 \\
      &=  (\ogamma - t)(y + \oGamma \xvar) = (\ogamma - t)\oGamma & & \text{by \cref{yinterm}} .
    \end{alignat*}
    Thus, $p$ violates \cref{eq:nl2} and hence $p$ fulfills
    \cref{eq:yz1} with strict inequality. Furthermore, due to the
    assumptions of this case, we can assume $\uGamma z - y = 0$.
  \item[\cref{eq:yz2}] Easily checked directly.
  \item[\cref{eq:us1}] As $\uGamma z - y = 0$, $z > 0$ and $\xvar > 0$,
    $p$ satisfies \cref{eq:us1} with strict inequality:
    \begin{align*}
      u - \ugamma \xvar = -y - \ugamma \xvar = -\uGamma z - \ugamma \xvar > 0 .
    \end{align*}
  \item[\cref{eq:ust2}] Follows from
    \[ \ogamma \xvar^{\epsilon}_i - u^{\epsilon}_i = \ogamma \xvar -\ogamma
    \delta^{\epsilon}_i - u + \uGamma \delta^{\epsilon}_i \leq \ogamma
    - t + \delta^{\epsilon}_i(\uGamma - \ogamma) = \ogamma -
    t^{\epsilon}_i . \]
  \item[\cref{eq:nl2}] Because $z > 0$ and $y = \uGamma z < 0$, it follows that
    $y^{\epsilon}_i < 0$ and $p_i^{\epsilon}$ is not subject to \cref{eq:nl2}.
  \end{description}
\end{proof}

\begin{lemma}
  \label{umore}
  Suppose $\ogamma > 0$ and $\uGamma < 0 < \oGamma$. Let $p = (\xvar,u,y,z,t) \in R^1$ with $0 < \xvar < 1$ and $\ugamma < t <
  \ogamma$. If $u > \xvar t$ and either $y \leq 0$ or $p$ satisfies \cref{eq:nl2} with strict inequality, then $p$ is not an extreme point of $R^1$.
\end{lemma}

We first comment that the assumption that either $y \leq 0$ or $p$ satisfies \cref{eq:nl2} with strict inequality
follows from the assumption $u > \xvar t$ and \cref{c1:ustslack} when $\ugamma < 0$. However, we state the assumption in
this way in order to make the applicability of this proposition clear for a later case when $\ugamma > 0$.

\begin{proof}
  This proof has the same structure as the proof of \cref{uless}.
  By \cref{ustslack}, $u > \xvar t$ implies that $p$ satisfies
  \cref{eq:us1,eq:ust2} with strict inequality. Also, by assumption,
  if $y > 0$, then $p$ satisfies \cref{eq:nl2} with strict inequality.
  It remains to show that the points $p_i^{\epsilon}$ satisfy
  \cref{eq:yu,eq:zs,eq:yz1,eq:yz2,eq:us2,eq:ust1,eq:nl1} for
  $\epsilon$ small enough.
  We consider four cases.

  \proofcase{a}{$y+u < 0$ and $z+\xvar=1$}
  Note that $z+\xvar=1$ and $\xvar < 1$ implies that $z > 0$.
  Thus, either \cref{eq:yz1} or \cref{eq:yz2} is satisfied strictly
  by $p$.
  If $y - \oGamma z <
  0$, define $y_i^{\epsilon} \eqdef (1-\alphep_i)\uGamma + \alphep_i y$, and otherwise, if $\uGamma z - y < 0$, define
  $y_i^{\epsilon} \eqdef (1- \alphep_i)\oGamma + \alpha_i y$ for $\epsilon >
  0$.  Then, for $\epsilon > 0$, define
  $p_i^{\epsilon} = (\xvar_i^{\epsilon}, u_i^{\epsilon}, y_i^\epsilon, z_i^\epsilon, t_i^{\epsilon})$
  where, for $i=1,2$,
  \begin{alignat*}{5}
    \xvar_i^{\epsilon} & = \alphep_i \xvar, \ \  &
    u_i^{\epsilon} & = \alphep_i u, \ \ &
    z_i^{\epsilon} & = (1-\alphep_i) + \alphep_i z, \ \  &
    t_i^{\epsilon} & = (1-\alphep_i) \ugamma + \alphep_i t .
  \end{alignat*}

  \begin{description}[nosep,leftmargin=2.5em]
  \item[\cref{eq:yu}] Satisfied strictly by $p$ by the assumption of
    this case.
  \item[\cref{eq:zs}] Easily checked directly.
  \item[\cref{eq:yz1}, \cref{eq:yz2}] Recall that either
    \cref{eq:yz1} or \cref{eq:yz2} is satisfied strictly by $p$. In
    the case $y - \oGamma z < 0$, \ie \cref{eq:yz1} is satisfied
    strictly, we only need to check $p_i^{\epsilon}$ satisfies \cref{eq:yz2}:
    \[ \uGamma z_i^{\epsilon} - y_i^{\epsilon} = \uGamma((1-\alphep_i)
    + \alphep_i z) - ((1-\alphep_i)\uGamma + \alphep_i y) = \alphep_i
    (\uGamma z - y) \leq 0 . \]
    On the other hand, if $\uGamma z - y <
    0$, \ie \cref{eq:yz2} is satisfied strictly, then
    \[ y_i^{\epsilon} - \oGamma z_i^{\epsilon} = \oGamma(1-\alphep_i)
    + \alphep_i y - \oGamma((1-\alphep_i) + \alphep_i z) = \alphep_i (
    y - \oGamma z ) \leq 0 . \]
  \item[\cref{eq:us2}] Easily checked directly.
  \item[\cref{eq:ust1}] Shown directly by
    \[ u_i^{\epsilon} - \ugamma \xvar_i^{\epsilon} = \alphep_i(u - \ugamma \xvar) \leq \alphep_i (t - \ugamma) = t_i^{\epsilon} - (1
    - \alphep_i)\ugamma - \alphep_i \ugamma = t_i^{\epsilon} - \ugamma . \]
  \item[\cref{eq:nl1}] Shown directly by
    \begin{align*}
      (u_i^{\epsilon} - \uGamma \xvar_i^{\epsilon})(u_i^{\epsilon} - \ugamma \xvar_i^{\epsilon})
      & = (\alphep_i)^2 (u - \uGamma \xvar)(u - \ugamma \xvar) \\
      & \leq  (\alphep_i)^2 (-\uGamma ) \xvar(t - \ugamma)
       = -\uGamma ( \xvar_i^{\epsilon} \alphep_i (t - \ugamma))
        = -\uGamma \xvar_i^{\epsilon} (t_i^{\epsilon} - \ugamma) .
    \end{align*}
  \end{description}

  \proofcase{b}{$z+\xvar<1$}
  For $\epsilon > 0$, define
  $p_i^{\epsilon} = (\xvar_i^{\epsilon}, u_i^{\epsilon}, y_i^\epsilon, z_i^\epsilon, t_i^{\epsilon})$
  where, for $i=1,2$,
  \begin{alignat*}{5}
    \xvar_i^{\epsilon} & \eqdef \alphep_i \xvar, \ \ &
    u_i^{\epsilon} & \eqdef \alphep_i u, \ \  &
    y_i^{\epsilon} & \eqdef \alphep_i y, \ \ &
    z_i^{\epsilon} & \eqdef \alphep_i z, \ \ &
    t_i^{\epsilon} & \eqdef  (1-\alphep_i)\ugamma + \alphep_i t .
  \end{alignat*}
  It is clear that \cref{eq:yu,eq:yz1,eq:yz2,eq:us1,eq:us2}
  are satisfied by $p_i^{\epsilon}$ for $i=1,2$. By the assumption of
  this case \cref{eq:zs} is strictly satisfied by $p$. The remaining
  inequalities \cref{eq:ust1,eq:nl1} depend only on the
  variables $\xvar$, $u$, and $t$, and the definitions of
  $u_i^{\epsilon}$,$\xvar_i^{\epsilon}$, and $t_i^{\epsilon}$ are the same
  as in Case (a).

  \proofcase{c}{$y+u=0$, $z+\xvar=1$, and $y - \oGamma z < 0$}
  For $\epsilon > 0$, define
  $p_i^{\epsilon} = (\xvar_i^{\epsilon}, u_i^{\epsilon}, y_i^\epsilon, z_i^\epsilon, t_i^{\epsilon})$
  where, for $i=1,2$,
  \begin{alignat*}{5}
    \xvar_i^{\epsilon} & \eqdef \alphep_i \xvar, \ \ &
    u_i^{\epsilon} & \eqdef \alphep_i u, \ \ &
    y_i^{\epsilon} & \eqdef \alphep_i y, \ \ &
    z_i^{\epsilon} & \eqdef (1-\alphep_i) + \alphep_i z, \ \ &
    t_i^{\epsilon} & \eqdef  (1-\alphep_i)\ugamma + \alphep_i t .
  \end{alignat*}

\begin{description}[nosep,leftmargin=2.5em]
  \item[\cref{eq:yu}, \cref{eq:zs}] Easily checked directly.
  \item[\cref{eq:yz1}] Satisfied strictly by $p$ by the assumption of
    this case.
  \item[\cref{eq:yz2}] We show that \cref{eq:yz2} is satisfied
    strictly by $p$. Indeed, if $\uGamma z - y = 0$, then the other
    equations for this case imply that $u - \uGamma \xvar = -\uGamma$.
    Then, evaluating $p$ in the left-hand-side of \cref{eq:nl1}
    yields:
    \[ (u - \uGamma \xvar)(u - \ugamma \xvar) = -\uGamma (u - \ugamma \xvar) >
    -\uGamma (\xvar t - \ugamma \xvar) = -\uGamma \xvar (t - \ugamma) \] %
    and so $p$ violates \cref{eq:nl1}.
  \item[\cref{eq:us2}] Easily checked directly.
  \item[\cref{eq:ust1}, \cref{eq:nl1}] As the definitions of
    $\xvar_i^{\epsilon}$, $u_i^{\epsilon}$, and $t_i^{\epsilon}$ are the
    same as in Case (a), the arguments in that case imply
    $p_i^{\epsilon}$ satisfies \cref{eq:ust1} and \cref{eq:nl1} for
    $\epsilon \in (0,1)$.
  \end{description}

  \proofcase{d}{$y+u=0$, $z+\xvar=1$, and $y - \oGamma z = 0$}
  For $\epsilon > 0$, define $p_i^{\epsilon} = (\xvar_i^{\epsilon},
  u_i^{\epsilon}, y_i^\epsilon, z_i^\epsilon, t_i^{\epsilon})$ where, for $i=1,2$,
  \begin{alignat*}{5}
    \xvar_i^{\epsilon} & \eqdef \xvar - \delta_i, \ \ &
    u_i^{\epsilon} & \eqdef u - \delta_i \oGamma, \ \ &
    y_i^{\epsilon} & \eqdef y + \delta_i \oGamma, \ \ &
    z_i^{\epsilon} & \eqdef z + \delta_i, \ \ &
    t_i^{\epsilon} & \eqdef t - \delta_i (\oGamma - \ugamma) .
  \end{alignat*}

  \begin{description}[nosep,leftmargin=2.5em]
  \item[\cref{eq:yu}, \cref{eq:zs}, \cref{eq:yz1}] Easily checked directly.
  \item[\cref{eq:yz2}] \cref{eq:yz2} is satisfied
    strictly by $p$ by the same argument as in the previous case.
  \item[\cref{eq:us2}] We show that $\ogamma \xvar - u > 0$, i.e.,
    \cref{eq:us2} is satisfied strictly by $p$. Indeed, the three
    equations in this case imply that $\oGamma \xvar - u = \oGamma$. Thus,
  \[ \ogamma \xvar - u = \ogamma \xvar - \oGamma \xvar + \oGamma = \ogamma \xvar + (1-\xvar)\oGamma > 0 . \]
  \item[\cref{eq:ust1}] Shown directly by
    \begin{align}
      u_i^{\epsilon} - \ugamma \xvar_i^{\epsilon} = u - \delta^{\epsilon}_i \oGamma - \ugamma (\xvar - \delta^{\epsilon}_i)
      &= u - \ugamma \xvar - \delta_i (\oGamma - \ugamma) \nonumber \\
      &\leq t - \ugamma - \delta_i (\oGamma - \ugamma) = t_i - \ugamma  . \label{uineq}
    \end{align}
  \item[\cref{eq:nl1}] As $y = \oGamma z$ and $z > 0$, this implies
    $y > 0$ and in turn $u < 0$. Thus,
    $ u - \uGamma \xvar < -\uGamma \xvar$ and so, for $\epsilon > 0$ small
    enough, also $u_i^{\epsilon} - \uGamma \xvar_i^{\epsilon} < -\uGamma
    \xvar_i^{\epsilon}$. Combining this with \cref{uineq} yields
    \[ (u_i^{\epsilon} - \uGamma \xvar_i^{\epsilon})( u_i^{\epsilon} -
    \ugamma \xvar_i^{\epsilon} ) \leq -\uGamma \xvar_i^{\epsilon} (t_i -
    \ugamma) . \]
  \end{description}
\end{proof}

\subsection{Proof of \cref{thm:case2}}
\label{sec:prov_case2}

We now state and prove the two main lemmas that support the proof of \cref{thm:case2}.

\begin{lemma}
  \label{c2:uless}
  Suppose $\ugamma < \ogamma < 0$ and $\uGamma < 0 < \oGamma$. Let $p = (\xvar,u,y,z,t) \in R^2$ with $0 < \xvar < 1$ and $\ugamma < t < \ogamma$. If $u < \xvar t$, then $p$ is not an extreme point of $R^2$.
\end{lemma}
\begin{proof}
  First, we show that $p$ satisfies \cref{c2:lin} with strict inequality.
  Observe that the inequality \cref{case2nlform} is valid for any point in $R^2$. Thus,
  \begin{align*}
    (\ogamma-\ugamma)y + \ugamma (\ogamma \xvar - u)  + \oGamma (u -
                                                       \ugamma \xvar)
     &< (\ogamma-\ugamma)y + \ugamma (\ogamma \xvar - \xvar t) +
      \oGamma (\xvar t - \ugamma \xvar) \\
    &= (\ogamma - \ugamma)y + (\ogamma - t)\xvar \ugamma + (t - \ugamma)\oGamma \xvar \leq  (t - \ugamma)\oGamma .
  \end{align*}
  When $u < \xvar t$, the inequality \cref{c2:lin} is satisfied with strict
  inequality, just as \cref{eq:nl1} is satisfied by strict inequality
  when $u < \xvar t$ and $\ogamma > 0$. As the substitution of \cref{c2:lin}
  for \cref{eq:nl1} is the only difference between the sets $R^2$ and
  $R^1$, the arguments of \cref{uless} apply directly to this case, and
  we can conclude that if $u < \xvar t$, $0 < \xvar < 1$, and $\ugamma < t <
  \ogamma$, then $p$ is not an extreme point of $R^2$.
  \end{proof}

\begin{lemma}
  \label{c2:umore}
  Suppose $\ugamma < \ogamma < 0$ and $\uGamma < 0 < \oGamma$. Let $p = (\xvar,u,y,z,t) \in R^2$ with $0 < \xvar < 1$ and $\ugamma < t < \ogamma$. If $u > \xvar t$, then $p$ is not an extreme point of $R^2$.
\end{lemma}
\begin{proof}
  This proof has the same structure as the proof of \cref{uless}.
  First, by \cref{ustslack}, $u > \xvar t$ implies that $p$ satisfies
  \cref{eq:us1,eq:ust2} with strict inequality, and by
  \cref{c1:ustslack} if also $y > 0$, then $p$ satisfies
  \cref{eq:nl2} with strict inequality. Also, as $\ogamma < 0$, it
  follows from $u \leq \ogamma \xvar$ and $\xvar > 0$ that $u < 0$. It remains
  to show that the points $p_i^{\epsilon}$ are feasible for the
  inequalities \cref{eq:yu,eq:zs,eq:yz1,eq:yz2,eq:us2,eq:ust1,c2:lin}.
  We consider four cases.

  \proofcase{a}{$y+u < 0$}
  For $\epsilon > 0$, define
  $p_i^{\epsilon} = (\xvar_i^{\epsilon}, u_i^{\epsilon}, y_i^\epsilon, z_i^\epsilon, t_i^{\epsilon})$
  where, for $i=1,2$,
  \begin{alignat*}{5}
    \xvar_i^{\epsilon} & \eqdef (1-\alphep_i) + \alphep_i \xvar, \ \ &
    u_i^{\epsilon} & \eqdef \ogamma(1-\alphep_i) + \alphep_i u, \  \ &
    y_i^{\epsilon} & \eqdef \alphep_i y, \ \\
    z_i^{\epsilon} & \eqdef \alphep_i z, \ \ &
    t_i^{\epsilon} & \eqdef (1-\alphep_i) \ogamma + \alphep_i t .
  \end{alignat*}

  \begin{description}[nosep,leftmargin=2.5em]
  \item[\cref{eq:yu}] Satisfied strictly by $p$ by the assumption of
    this case.
  \item[\cref{eq:zs,eq:yz1,eq:yz2}] Easily checked
    directly.
  \item[\cref{eq:us2}] Shown directly by
    \begin{equation}
      \label{c2:interm1}
      \ogamma \xvar_i^{\epsilon} - u_i^{\epsilon} = \ogamma (1 -\alphep_i) + \ogamma \alphep_i \xvar - (1-\alphep_i)\ogamma -
      \alphep_i u = \alphep_i (\ogamma \xvar - u) \geq 0 .
    \end{equation}
  \item[\cref{eq:ust1}] Shown directly by
    \begin{align}
      u_i^{\epsilon} - \ugamma \xvar_i^{\epsilon} &= \ogamma (1-\alphep_i)
      + \alphep_i u - \ugamma (1-\alphep_i) - \ugamma
      \alphep_i \xvar \nonumber \\
      &= (\ogamma - \ugamma) (1-\alphep_i) + \alphep_i(u - \ugamma \xvar)  \label{c2:interm2} \\
      &\leq (\ogamma - \ugamma) (1-\alphep_i) + \alphep_i (t - \ugamma) \nonumber \\
      & = (\ogamma - \ugamma)(1-\alphep_i) + t_i^{\epsilon} - (1-\alphep_i)\ogamma - \alphep_i \ugamma
       = t_i^{\epsilon} - \ugamma . \label{c2:interm4}
    \end{align}
  \item[\cref{c2:lin}] Using \cref{c2:interm1,c2:interm2},
    we get
    \begin{align*}
      & (\ogamma - \ugamma) y_i^{\epsilon} + \ugamma (\ogamma
      \xvar_i^{\epsilon} - u_i^{\epsilon}) + \oGamma (u_i^{\epsilon} -
      \ugamma \xvar_i^{\epsilon}) \\
      & = \alphep_i \Bigl( (\ogamma - \ugamma) y + \ugamma (\ogamma \xvar
      - u) + \oGamma (u - \ugamma \xvar)\Bigr) + \oGamma (\ogamma
      - \ugamma)(1-\alphep_i) \\
      & \leq \alphep_i \oGamma (t - \ugamma) + (1-\alphep_i) \oGamma
      (\ogamma - \ugamma) = \oGamma (t_i^{\epsilon} - \ugamma)
    \end{align*}
    where the last equation follows from \cref{c2:interm4}.
  \end{description}

  \proofcase{b}{$z+\xvar<1$}
  For $\epsilon > 0$, define
  $p_i^{\epsilon} = (\xvar_i^{\epsilon}, u_i^{\epsilon}, y_i^\epsilon, z_i^\epsilon, t_i^{\epsilon})$
  where, for $i=1,2$,
  \begin{alignat*}{5}
    \xvar_i^{\epsilon} & \eqdef \alphep_i \xvar, \ \ &
    u_i^{\epsilon} & \eqdef \alphep_i u, \ \ &
    y_i^{\epsilon} & \eqdef \alphep_i y, \ \ &
    z_i^{\epsilon} & \eqdef \alphep_i z, \ \ &
    t_i^{\epsilon} & \eqdef  \alphep_i t + (1-\alphep_i)\ugamma .
  \end{alignat*}

  \begin{description}[nosep,leftmargin=2.5em]
  \item[\cref{eq:yu}] Easily checked directly.
  \item[\cref{eq:zs}] Satisfied strictly by $p$ by the assumption of
    this case.
  \item[\cref{eq:yz1}, \cref{eq:yz2}, \cref{eq:us2}] Easily checked
    directly.
  \item[\cref{eq:ust1}] Shown directly by
    \begin{align}
      u_i^{\epsilon}  - \ugamma \xvar_i^{\epsilon} = \alphep_i u - \ugamma \alphep_i \xvar
      & \leq \alphep_i (t - \ugamma) = t_i^{\epsilon} - (1-\alphep_i)\ugamma - \alphep_i \ugamma \nonumber\\
      & = t_i^{\epsilon} - \ugamma . \label{c22:interm}
    \end{align}
  \item[\cref{c2:lin}] Shown directly by
    \begin{align*}
      (\ogamma-\ugamma)y_i^{\epsilon} & + \ugamma (\ogamma
      \xvar_i^{\epsilon} - u_i^{\epsilon}) + \oGamma (u_i^{\epsilon} -
                                        \ugamma \xvar_i^{\epsilon}) \\
      & = \alphep_i \bigl( (\ogamma-\ugamma)y
      + \ugamma (\ogamma \xvar - u) +
        \oGamma (u - \ugamma \xvar) \bigr)
       \leq \alphep_i \oGamma(t - \ugamma) = \oGamma (t_i^{\epsilon} - \ugamma)
    \end{align*}
    where the last equation follows as in \cref{c22:interm}.
  \end{description}

  \proofcase{c}{$y - \oGamma z < 0$ and $\uGamma z - y < 0$}
  For $\epsilon > 0$, define
  $p_i^{\epsilon} = (\xvar_i^{\epsilon}, u_i^{\epsilon}, y_i^\epsilon, z_i^\epsilon, t_i^{\epsilon})$
  where, for $i=1,2$,
  \begin{alignat*}{5}
    \xvar_i^{\epsilon} & \eqdef \alphep_i \xvar, \ \ &
    u_i^{\epsilon} & \eqdef \alphep_i u, \ \  &
    y_i^{\epsilon} & \eqdef \alphep_i y, \ \ &
    z_i^{\epsilon} & \eqdef (1-\alphep_i) + \alphep_i z, \ \ &
    t_i^{\epsilon} & \eqdef  \alphep_i t + (1-\alphep_i)\ugamma.
  \end{alignat*}
  Then, it is easily seen by construction that $p_i^{\epsilon}$ satisfies \cref{eq:zs} for any $\epsilon \in (0,1)$,
  $i=1,2$.
  As the definitions of $\xvar_i^{\epsilon}$, $u_i^{\epsilon}$, $y_i^{\epsilon}$, and $t_i^{\epsilon}$ are the same as in Case (b),
  the arguments of Case (b) apply for all inequalities that do not contain the variable $z$. This just leaves
  and \cref{eq:yz1,eq:yz2}, which by assumption are satisfied strictly by $p$, and so the proof for this case
  is complete.

  \proofcase{d}{$y+u = 0$, $z+\xvar=1$, and either $y - \oGamma z = 0$ or $\uGamma z - y = 0$}
  For $\epsilon > 0$, define
  $p_i^{\epsilon} = (\xvar_i^{\epsilon}, u_i^{\epsilon}, y_i^\epsilon, z_i^\epsilon, t_i^{\epsilon})$
  where, for $i=1,2$,
  \begin{alignat*}{5}
    \xvar_i^{\epsilon} & \eqdef (1-\alphep_i) + \alphep_i \xvar, \ \ &
    u_i^{\epsilon} & \eqdef \alphep_i u, \ \ &
    y_i^{\epsilon} & \eqdef \alphep_i y, \ \ &
    z_i^{\epsilon} & \eqdef \alphep_i z, \ \ &
    t_i^{\epsilon} & \eqdef (1-\alphep_i)\frac{\ugamma \ogamma}{\oGamma} + \alphep_i t.
  \end{alignat*}

  \begin{description}[nosep,leftmargin=2.5em]
  \item[\cref{eq:yu,eq:zs,eq:yz1,eq:yz2}] Easily checked directly.
  \item[\cref{eq:us2}] We show that $p$ satisfies \cref{eq:us2}
    strictly. Suppose for purpose of contradiction that $\ogamma \xvar - u
    = 0$. Then,
    \begin{align*}
      (\ogamma-\ugamma)y + \ugamma (\ogamma \xvar - u) + \oGamma (u -
      \ugamma \xvar) & = (\ogamma-\ugamma)y + \oGamma (\ogamma \xvar - \ugamma
      \xvar)\\
      & = (\ogamma - \ugamma)(y + \oGamma \xvar)
       = (\ogamma - \ugamma)\oGamma
       > \oGamma (t - \ugamma)
    \end{align*}
    where we have used $y + \oGamma \xvar =
    \oGamma z + \oGamma \xvar = \oGamma$. Thus, when $\ogamma \xvar - u = 0$
    then \cref{c2:lin} is violated, and hence we conclude that
    \cref{eq:us2} is satisfied strictly by $p$.
  \item[\cref{eq:ust1}] We show that $p$ satisfies \cref{eq:ust1}
    strictly. Indeed, as $y = -u$, we find that
    \[ (\ogamma - \ugamma) y + \ugamma (\ogamma \xvar - u) = (\ogamma -
    \ugamma)(-u) + \ugamma (\ogamma \xvar - u) = \ogamma (\ugamma \xvar - u) >
    0 \] since $\ogamma < 0$ and $\ugamma \xvar - u < 0$. Thus,
    rearranging inequality \cref{c2:lin} yields
    \[ u - \ugamma \xvar \leq t - \ugamma - \frac{1}{\oGamma} \bigl(
    (\ogamma - \ugamma) y + \ugamma (\ogamma \xvar - u) \bigr) < t -
    \ugamma \] which shows \cref{eq:ust1} is satisfied strictly by
    $p$.
  \item[\cref{c2:lin}] Shown directly by
    \begin{align*}
      & (\ogamma-\ugamma)y_i^{\epsilon} + \ugamma (\ogamma
      \xvar_i^{\epsilon} - u^i_{\epsilon}) + \oGamma (u_i^{\epsilon} -
      \ugamma
      \xvar_i^{\epsilon}) \\
      &= \alphep_i \bigl((\ogamma-\ugamma)y + \ugamma (\ogamma \xvar - u)
      + \oGamma (u - \ugamma \xvar) \bigr) + \ugamma \ogamma
      (1-\alphep_i) - \oGamma \ugamma (1-\alphep_i) \\
      &\leq \alphep_i \oGamma(t - \ugamma) - (1-\alphep_i)\ugamma (\oGamma - \ogamma) \\
      &= \oGamma \Bigl( t_i^{\epsilon} - (1-\alphep_i)\frac{\ugamma \ogamma}{\oGamma} - \ugamma \alphep_i \Bigr) - (1-\alphep_i)\ugamma (\oGamma - \ogamma) \\
      &= \oGamma (t_i^{\epsilon} - \ugamma) - (1-\alphep_i)(\ugamma \ogamma) + (1-\alphep_i) (\ugamma \ogamma)
      =\oGamma (t_i^{\epsilon} - \ugamma) .
    \end{align*}
  \end{description}
\end{proof}

\subsection{Proof of \cref{thm:case3}}
\label{sec:prov_case3}

We now state and prove the two main lemmas that support the proof of
\cref{thm:case3}. We prepare the proofs with the following
proposition.
\begin{proposition}
  \label{c3:nlvalid}
  Let $\uGamma < 0$. If $p \in R^3$, then $p$ satisfies the following inequality:
  \begin{align}
    (\ugamma - \uGamma)\xvar &\leq -\uGamma \label{c3:nl1}
  \end{align}
  In addition, if $p$ satisfies \cref{eq:yu,eq:zs,eq:yz2,eq:us1} at equality, then it
  satisfies \cref{c3:nl1} at equality.
\end{proposition}
\begin{proof}
  First, aggregating \cref{eq:yu} with weight $1$, \cref{eq:zs} with weight $-\uGamma$, \cref{eq:yz2} with
  weight $1$, and \cref{eq:us1} with weight $-1$, yields \cref{c3:nl1}.
  If \cref{eq:yu,eq:zs,eq:yz2,eq:us1} are all
  satisfied at equality, then $p$ satisfies \cref{c3:nl1} at equality.
\end{proof}

\begin{lemma}
  \label{c3:uless}
  Suppose $0 < \ugamma < \ogamma$ and $\uGamma < 0 < \oGamma$. Let $p
  = (\xvar,u,y,z,t) \in R^3$ with $0 < \xvar < 1$ and $\ugamma < t <
  \ogamma$. If $u < \xvar t$, then $p$ is not an extreme point of $R^3$.
\end{lemma}
\begin{proof}
  This proof has the same structure as the proof of \cref{uless}.
  By \cref{ustslack}, $p$ satisfies \cref{eq:us2,eq:ust1,eq:nl1} with
  strict inequality. Also, as $u \geq \ugamma \xvar > 0$, \cref{eq:yu}
  implies that $y < 0 \leq \oGamma z$, and hence $p$ satisfies
  \cref{eq:yz1} with strict inequality. In addition, by
  \cref{c3:lin},
  \[ \ogamma \xvar - u \leq \frac{-\uGamma}{\ugamma - \uGamma}(\ogamma -
  t) < \ogamma - t \] as $\ogamma - t > 0$ and $\ugamma - \uGamma >
  -\uGamma$ because $\ugamma > 0$, and so $p$ satisfies
  \cref{eq:ust2} with strict inequality. It remains to show that the
  points $p_i^{\epsilon}$ satisfy \cref{eq:yu,eq:zs,eq:yz2,eq:us1,c3:lin} for $\epsilon$ small enough.
  We consider four cases.

  \proofcase{a}{$y+u < 0$}
  For $\epsilon > 0$, define
  $p_i^{\epsilon} = (\xvar_i^{\epsilon}, u_i^{\epsilon}, y_i^\epsilon, z_i^\epsilon, t_i^{\epsilon})$
  where, for $i=1,2$,
  \begin{alignat*}{5}
    \xvar_i^{\epsilon} & \eqdef \alphep_i \xvar, \ \ &
    u_i^{\epsilon} & \eqdef \alphep_i u, \ \ &
    y_i^{\epsilon} & \eqdef (1-\alphep_i)\uGamma + \alphep_i y, \ \ \\
    z_i^{\epsilon} & \eqdef (1-\alphep_i) + \alphep_i z, \ \ &
    t_i^{\epsilon} & \eqdef (1-\alphep_i) \ogamma + \alphep_i t .
  \end{alignat*}
  \begin{description}[nosep,leftmargin=2.5em]
  \item[\cref{eq:yu}] Satisfied strictly by $p$ by the assumption of
    this case.
  \item[\cref{eq:zs}, \cref{eq:yz2}, \cref{eq:us1}] Easily checked directly.
  \item[\cref{c3:lin}] Shown directly by
    \begin{align*}
      (\ugamma - \uGamma)(\ogamma \xvar_i^{\epsilon} - u_i^{\epsilon})  &= \alphep_i (\ugamma - \uGamma)(\ogamma \xvar - u) \\
      & \leq \alphep_i(-\uGamma)(\ogamma - t)
       = -\uGamma ( \alphep_i \ogamma - t_i^{\epsilon} +
      (1-\alphep_i)\ogamma) = -\uGamma (\ogamma - t_i^{\epsilon}) .
    \end{align*}
  \end{description}

  \proofcase{b}{$z+\xvar < 1$}
  For $\epsilon > 0$, define
  $p_i^{\epsilon} = (\xvar_i^{\epsilon}, u_i^{\epsilon}, y_i^\epsilon, z_i^\epsilon, t_i^{\epsilon})$
  where, for $i=1,2$,
  \begin{alignat*}{5}
    \xvar_i^{\epsilon} & \eqdef \alphep_i \xvar, \ \ &
    u_i^{\epsilon} & \eqdef \alphep_i u, \ \ &
    y_i^{\epsilon} & \eqdef \alphep_i y, \ \ &
    z_i^{\epsilon} & \eqdef \alphep_i z, \ \ &
    t_i^{\epsilon} & \eqdef (1-\alphep_i) \ogamma + \alphep_i t .
  \end{alignat*}
  Then, $p_i^{\epsilon}$ is easily seen to satisfy
  \cref{eq:yu,eq:yz2,eq:us1} for any $\epsilon
  \in (0,1)$. \cref{eq:zs} is satisfied strictly by the assumption of
  this case. In addition, as the definitions of $u_i^{\epsilon}$,
  $\xvar_i^{\epsilon}$, and $t_i^{\epsilon}$ are the same as in Case (a),
  \cref{c3:lin} is satisfied by $p_i^\epsilon$ for $i=1,2$ and any
  $\epsilon \in (0,1)$.

  \proofcase{c}{$y > \uGamma z$}
  For $\epsilon > 0$, define
  $p_i^{\epsilon} = (\xvar_i^{\epsilon}, u_i^{\epsilon}, y_i^\epsilon, z_i^\epsilon, t_i^{\epsilon})$
  where, for $i=1,2$,
  \begin{alignat*}{5}
    \xvar_i^{\epsilon} & \eqdef \alphep_i \xvar, \ \ &
    u_i^{\epsilon} & \eqdef \alphep_i u, \ \ &
    y_i^{\epsilon} & \eqdef \alphep_i y, \ \ &
    z_i^{\epsilon} & \eqdef (1-\alphep_i) + \alphep_i z, \ \ &
    t_i^{\epsilon} & \eqdef (1-\alphep_i) \ogamma + \alphep_i t.
  \end{alignat*}
  Then, $p_i^{\epsilon}$ is
  easily seen to satisfy \cref{eq:yu,eq:zs,eq:us1} for any $\epsilon
  \in (0,1)$. \cref{eq:yz2} is satisfied strictly by the assumption
  of this case. In addition, as the definitions of $u_i^{\epsilon}$,
  $\xvar_i^{\epsilon}$, and $t_i^{\epsilon}$ are the same as in Case (a),
  \cref{c3:lin} is satisfied by $p_i^\epsilon$ for $i=1,2$ and any
  $\epsilon \in (0,1)$.

  \proofcase{d}{$y+u=0$, $z+\xvar=1$, and $y=\uGamma z$}
  For $\epsilon > 0$, define
  $p_i^{\epsilon} = (\xvar_i^{\epsilon}, u_i^{\epsilon}, y_i^\epsilon, z_i^\epsilon, t_i^{\epsilon})$
  where, for $i=1,2$,
  \begin{alignat*}{5}
    \xvar_i^{\epsilon} & \eqdef \xvar- \delta^{\epsilon}_i, \ \ &
    u_i^{\epsilon} & \eqdef u - \uGamma \delta_i, \ \ &
    y_i^{\epsilon} & \eqdef y + \uGamma \delta_i, \ \ &
    z_i^{\epsilon} & \eqdef z + \delta_i, \ \ &
    t_i^{\epsilon} & \eqdef t + \delta_i \frac{(\ugamma - \uGamma)(\ogamma - \uGamma)}{(-\uGamma)} .
  \end{alignat*}
  \begin{description}[nosep,leftmargin=2.5em]
  \item[\cref{eq:yu}, \cref{eq:zs}, \cref{eq:yz2}] Easily checked directly.
  \item[\cref{eq:us1}] We show that \cref{eq:us1} is satisfied
    strictly by $p$. Indeed, suppose to the contrary that $\ugamma \xvar-
    u = 0$. Then, by \cref{c3:nlvalid}, $(\ugamma - \uGamma)\xvar=
    -\uGamma$. Thus, using $u < \xvar t$, $\ugamma > 0$ and $-\uGamma > 0$,
    \[ (\ugamma - \uGamma)(\ogamma \xvar- u) > (\ugamma -
    \uGamma)(\ogamma \xvar- \xvar t) = -\uGamma (\ogamma - t) \] and hence
    \cref{c3:lin} is violated. Thus, \cref{eq:us1}
    is satisfied strictly by $p$.
  \item[\cref{c3:lin}] Since $u>\xvar t$ and \cref{c3:nlvalid} we
    show the validity of \cref{c3:lin} by
    \begin{align*}
      (\ugamma - \uGamma)(\ogamma \xvar_i^{\epsilon} - u_i^{\epsilon}) & =
      (\ugamma - \uGamma)\bigl(\ogamma (\xvar - \delta_i) - (u - \uGamma
      \delta_i)\bigr) \\
      &= (\ugamma - \uGamma)(\ogamma \xvar - u) - \delta_i(\ugamma - \uGamma)(\ogamma - \uGamma) \\
      &< (\ugamma - \uGamma)\xvar(\ogamma - t) - \delta_i(\ugamma - \uGamma)(\ogamma - \uGamma) \\
      &\leq -\uGamma(\ogamma - t) - \delta_i(\ugamma - \uGamma)(\ogamma - \uGamma) \\
      &= -\uGamma\Bigl(\ogamma - t_i + \delta_i \frac{(\ugamma - \uGamma)(\ogamma - \uGamma)}{(-\uGamma)}\Bigr) -
		\delta_i(\ugamma - \uGamma)(\ogamma - \uGamma) \\
      &= -\uGamma(\ogamma - t_i) .
    \end{align*}
  \end{description}
\end{proof}

\begin{lemma}
  \label{c3:umore}
  Suppose $0 < \ugamma < \ogamma$ and $\uGamma < 0 < \oGamma$. Let $p = (\xvar,u,y,z,t) \in R^3$ with $0 < \xvar < 1$ and $\ugamma < t < \ogamma$. If $u > \xvar t$, then $p$ is not an extreme point of $R^3$.
\end{lemma}
\begin{proof}
  Using $\ugamma - \uGamma > 0$, we have
  \[ (\ugamma - \uGamma)(\xvar \ogamma - u) < (\ugamma - \uGamma)\xvar(\ogamma - t) \leq -\uGamma (\ogamma - t) \]
  by \cref{c3:nl1} in \cref{c3:nlvalid}.

  When $u > \xvar t$, the inequality \cref{c3:lin} is satisfied with strict
  inequality, just as \cref{eq:nl2} is satisfied by strict inequality
  when $u > \xvar t$ and $\ugamma < 0$ as in Case 1. As the substitution of
  \cref{c3:lin} for \cref{eq:nl2} is the only difference between the
  sets $R^3$ and $R^1$, \cref{umore} applies directly to this case, and
  we can conclude that if $u > \xvar t$, $0 < \xvar < 1$, and $\ugamma < t <
  \ogamma$, then $p$ is not an extreme point of $R^3$.
\end{proof}


\section{Computational Results}\label{sec:computational_results}
In this section, we present results from computational
experiments conducted on instances from the literature and on larger randomly generated
instances derived from the instances in the literature. We show that the proposed inequalities indeed strengthen the
relaxation of the \pqformulation and are able to speed up the global
solution process, especially on sparse instances.

\subsection{Computational setup}

The experiments were conducted on a cluster with 64bit Intel Xeon
X5672 CPUs at \SI{3.2}{\giga\hertz} with \SI{12}{\mega\byte} cache and
\SI{48}{\mega\byte} main memory. To limit the impact of variability in machine performance, \eg by cache
misses, we run only one job on each node at a time.

The model is implemented in the GAMS language and processed with GAMS
version 24.7.1. The \pqformulation is solved to global optimality with
SCIP version 3.2 which used CPLEX 12.6.3 as LP solver and Ipopt 3.12
as local NLP solver. The relaxations, which are LPs or SOCPs, are
solved with CPLEX 12.6.3. We used the predefined timelimit of 1000
seconds and use a relative gap of $10^{-6}$ as termination criterion
(GAMS options $\text{\texttt{OPTCA}} = 0.0$ and $\text{\texttt{OPTCR}}
= 10^{-6}$).

\subsection{Adding the inequalities}

Recall that the initial step to construct the 5-variable relaxation
was to focus on a (Attribute, Pool, Output) tuple and extend the model by
the aggregated variables $u, \xvar, z, y, t$ for each such pair. We follow this approach in the implementation. We extend the model by
the aggregated variables and rely on the solver to replace or
disaggregate the variables in the constraints if it is considered
advantageous. We add the linear inequalities whenever they are valid
(specifically \cref{eq:beast_case2} is added whenever $\oGamma > 0$
and \cref{eq:beast_case3} is added whenever $\uGamma < 0$). Inequality
\cref{eq:ineq1} is second-order cone (SOC) representable and could in
principle be added directly as SOC constraint. However, 
we are not able to directly formulate \cref{yikes} using a linear or
second-order cone representation and we thus resort to a cutting
plane algorithm. Namely, whenever the relaxation solution has $y>0$
for a specific (Attribute, Pool, Output), a gradient inequality at this
point is separated and the relaxation is solved again. Note that the
gradient inequality is also valid for $y\leq 0$ due to
\cref{thm:convex_extension}. Since the gradient inequalities towards
the end of the separation loop become almost parallel, the interior
point SOCP solver frequently runs into numerical trouble. To
circumvent this, \cref{eq:ineq1} is not added directly to the model,
but linear gradient inequalities are also separated from all conic
inequalities in the same separation loop. The major advantage is that
all relaxations are then LPs and thus solved very efficiently. This
approach in our experience provides much better running times than
solving SOCP relaxations in the separation loop. 

We separate the inequalities only at the root node of the spatial
\branchAndBound algorithm. More precisely, we set up the separation
loop for both inequalities and separate until the absolute violation
of the conic inequality and inequality \eqref{yikes} are below $10^{-4}$
and $10^{-5}$, respectively. Then we pass the \pqformulation and all
inequalities that have been separated to \Scip and solve the problem
globally. The separation therefore does not make use of any model changes or
strengthening that \Scip performs during preprocessing or from propagations during
its own cutting plane loop.

In the following we use \pqrelaxation to refer to the McCormick
relaxation of the \pqformulation. The relaxation that arises by
strengthening the \pqrelaxation with our valid inequalities is called
\CJJJrelaxation.

\subsection{Instances}
\label{sec:instances}

We perform experiments on two sets of instances: The pooling instances
from the GAMSLIB~\cite{gamslib} and new instances that we randomly generated based on structures from the GAMSLIB
instances. The
GAMSLIB instances are encoded in the \texttt{pool} model as different
cases yielding  14 instances. All of them were first presented in
scientific publications about the pooling problem. It comprises three 
instances on the original network from
Haverly~\cite{haverly:78,haverly:79}. Furthermore, it contains
instances from the
publications~\cite{foulds.haugland.jornsten:92,ben-tal.eiger.gershovitz:94,adhya.tawarmalani.sahinidis:99:lagrangian_for_pooling,audet.et.al:04}.

The random instances are generated in the following way. The basis are
copies of the Haverly instances. The resulting disconnected graphs are
then supplemented by randomly adding a specific number of admissible
edges in a pooling network. As the resulting network might still be
disconnected, the first edges are chosen as to connect two
disconnected components until the graph is connected.
As the GAMSLIB includes three instances of the Haverly network with different
parameters, the distribution among the three Haverly instances is
sampled randomly. Next, for each copy a factor $\phi \in [0.5, 2]$ is
sampled uniformly and all concentration parameters, \ie
$\qualityAtInp$ and $\qualityUBAtOutp$ of that copy are scaled by
$\phi$. Lower bounds on the concentration are not used in these
instances but could be sampled and handled in the separation in a
similar way.

We generated instances with 10, 15, and 20 copies of the Haverly
network. The number of edges to be added are multiples of the number of
copies of the Haverly network. For each such pair of number of copies
and number of additional edges, we sample 10 instances. In total 180
instances are generated.

The new instances and the scripts to create them are available
online\footnote{\url{https://github.com/poolinginstances/poolinginstances},
  commit \texttt{e50a2c31ceed}}.

\subsection{Results}

\newcommand{\resultsfolder}{../results_random_haverly_7}

\begin{table}
  \centering
  \resizebox{\textwidth}{!}{
    \begin{tabular}{lrrrrrrrr}
\toprule
Instance & \multicolumn{2}{c}{Graph} & \multicolumn{2}{c}{\PQ} & \multicolumn{3}{c}{\CJJJ} & Opt\\
\cmidrule(lr){2-3}\cmidrule(lr){4-5}\cmidrule(lr){6-8}
 & \multicolumn{1}{c}{Nodes} & \multicolumn{1}{c}{Arcs} & \multicolumn{1}{c}{Absolute} & \multicolumn{1}{c}{Gap} & \multicolumn{1}{c}{Absolute} & \multicolumn{1}{c}{Gap} & \multicolumn{1}{c}{Closed} & \\
\midrule
 adhya1 & 11 & 13 & -766.3 & 39.4 \% & -697.0 & 26.8 \% & 32.0 \% & -549.8\\
 adhya2 & 11 & 13 & -570.8 & 3.8 \% & -568.3 & 3.4 \% & 11.8 \% & -549.8\\
 adhya3 & 15 & 20 & -571.3 & 1.8 \% & -570.7 & 1.7 \% & 6.3 \% & -561.0\\
 adhya4 & 15 & 18 & -961.2 & 9.5 \% & -955.4 & 8.9 \% & 7.0 \% & -877.6\\
 bental4 & 7 & 4 & -550.0 & 22.2 \% & -450.0 & 0.0 \% & 100.0 \% & -450.0\\
 haverly1 & 6 & 6 & -500.0 & 25.0 \% & -400.0 & 0.0 \% & 100.0 \% & -400.0\\
 haverly2 & 6 & 6 & -1000.0 & 66.7 \% & -600.0 & 0.0 \% & 100.0 \% & -600.0\\
 haverly3 & 6 & 6 & -800.0 & 6.7 \% & -791.7 & 5.6 \% & 16.7 \% & -750.0\\
\bottomrule
\end{tabular}

  }
  \caption{Results on GAMSLIB instances where the \pqformulation does
    not provide the optimum}
  \label{tab:gamslib}
\end{table}

First, we consider the 14 GAMSLIB instances. For six instances the
\pqrelaxation provides the optimal bound and hence these instances are
not considered anymore. \Cref{tab:gamslib} shows results on the
remaining eight GAMSLIB instances. Along with the size of the graph in terms
of number of nodes and arcs, \cref{tab:gamslib} presents the value of
the different relaxations and their gaps. The column ``Opt'' shows the
global optimum computed by solving the \nonconvex \pqformulation. The
instances are small, but the results are encouraging. Our
relaxation gives a stronger dual bound on all instances compared to
the \pqrelaxation and on three instances the gap is closed completely.


\begin{table}
  \centering
  \resizebox{\textwidth}{!}{
    \begin{tabular}{rrrrrrrrrrrr}
\toprule
Cop.  & \multicolumn{3}{c}{Graph} & \multicolumn{2}{c}{Gap [\%]} & \multicolumn{3}{c}{Global \PQ} & \multicolumn{3}{c}{Global \CJJJ}\\
\cmidrule(lr){2-4}\cmidrule(lr){5-6}\cmidrule(lr){7-9}\cmidrule(lr){10-12}
 & \multicolumn{1}{c}{$\card{\nodes}$} & \multicolumn{1}{c}{$\card{\arcs}$} & \multicolumn{1}{c}{$\card{\arcs^+}$} & \multicolumn{1}{c}{\PQ} & \multicolumn{1}{c}{\CJJJ} & \multicolumn{1}{c}{TL} & \multicolumn{1}{c}{Time} & \multicolumn{1}{c}{Nodes} & \multicolumn{1}{c}{TL} & \multicolumn{1}{c}{Time} & \multicolumn{1}{c}{Nodes}\\
\midrule
10 & 60 & 70 & 10 & 13.0 & 3.2 & 0 & 5.0 & 7098.1 & 0 & 1.8 & 213.9\\
 & 60 & 80 & 20 & 8.3 & 3.8 & 0 & 3.4 & 3011.6 & 0 & 3.4 & 727.0\\
 & 60 & 90 & 30 & 4.7 & 2.9 & 0 & 3.1 & 1397.2 & 0 & 3.2 & 371.4\\
 & 60 & 100 & 40 & 3.0 & 1.9 & 0 & 2.3 & 993.2 & 0 & 4.1 & 530.4\\
 & 60 & 110 & 50 & 2.6 & 2.1 & 0 & 3.3 & 1034.4 & 0 & 6.0 & 665.9\\
 & 60 & 120 & 60 & 3.3 & 2.4 & 0 & 6.3 & 2320.4 & 0 & 9.3 & 1511.6\\
15 & 90 & 105 & 15 & 10.6 & 3.2 & 0 & 63.0 & 106880.6 & 0 & 7.0 & 2023.2\\
 & 90 & 120 & 30 & 7.2 & 3.3 & 1 & 53.2 & 40031.5 & 1 & 20.0 & 3480.5\\
 & 90 & 135 & 45 & 4.9 & 3.4 & 1 & 36.6 & 24087.1 & 0 & 31.0 & 8202.0\\
 & 90 & 150 & 60 & 4.1 & 3.0 & 1 & 33.4 & 19234.3 & 0 & 24.0 & 5043.2\\
 & 90 & 165 & 75 & 3.3 & 2.5 & 0 & 21.2 & 13919.3 & 0 & 37.1 & 10576.1\\
 & 90 & 180 & 90 & 3.8 & 3.0 & 1 & 47.4 & 21998.2 & 1 & 51.8 & 8300.6\\
20 & 120 & 140 & 20 & 13.4 & 4.3 & 9 & 993.9 & 1655439.0 & 1 & 44.3 & 7327.0\\
 & 120 & 160 & 40 & 6.0 & 2.9 & 4 & 296.9 & 175642.9 & 3 & 116.8 & 18319.4\\
 & 120 & 180 & 60 & 4.5 & 2.8 & 5 & 287.6 & 84123.7 & 4 & 213.3 & 29497.7\\
 & 120 & 200 & 80 & 4.1 & 2.8 & 2 & 84.3 & 40476.0 & 2 & 68.5 & 12186.5\\
 & 120 & 220 & 100 & 3.1 & 2.3 & 3 & 159.1 & 44945.7 & 3 & 142.3 & 20325.7\\
 & 120 & 240 & 120 & 2.5 & 2.0 & 2 & 187.5 & 69610.8 & 3 & 224.1 & 36090.3\\
\midrule
\multicolumn{1}{l}{Total} & -- & -- & -- & 5.7 & 2.9 & 29 & 37.1 & 12685.7 & 18 & 25.0 & 3108.0\\
\bottomrule
\end{tabular}

  }
  \caption{Results on randomly generated instances}
  \label{tab:random_haverly_1}
\end{table}



\Cref{tab:random_haverly_1} presents results on the larger randomly generated
networks. The instances are grouped by the number of copies of the
Haverly network (first column) and by the number of edges that have
been added to the network (column $\card{\arcs^+}$). Each row thus
provides aggregated results over 10 instances. The last row represents
the total over all instances. The group of columns labeled with
``Graph'' shows statistics about the graphs. Besides the number of
random arcs added $\card{\arcs^+}$, the number of nodes
$\card{\nodes}$ and arcs $\card{\arcs}$ is shown. The numbers are
identical within each group of instances. Next, the average gap for the
\pqrelaxation and the \CJJJrelaxation is shown. For both approaches
the gap is computed \wrt the best known primal bound for the problem
and thus reflects only differences in the dual bound. Finally, the
last two groups of columns show statistics about the global solution
process using the \pqformulation and \CJJJrelaxation at the root. We
report number of instances that were terminated due to the time limit (column
``TL''), time, and number of nodes. For time and nodes, the shifted
geometric mean with shift 2 and 100, respectively, is used to
aggregate the results. Furthermore, only instances where both
approaches finished within the time limit are considered in the
computation of the number of nodes. For \CJJJ, the separation time for the
nonlinear inequalities is taken into account by adding it to the time
SCIP needed to solve the problem.

The \CJJJrelaxation is effective in reducing the root gap,
leaving an average gap of \SI{2.9}{\%} compared to the \SI{5.7}{\%} of
the \pqrelaxation. The \CJJJrelaxation
performs especially well on instances with sparse networks. This is expected, since the relaxation
provides the optimal dual bound on two of the three Haverly instances
(\cf \cref{tab:gamslib}) that we used to construct the random
instances. All but one instance of the testset experience an
improvement of the dual bounds due to the additional inequalities. The
most notable effect of the stronger root bound is on the number of
branch\&bound nodes needed to solve an instances to global optimality.
The shifted geometric mean of the nodes is reduced from 12685 to
3108, a reduction of \SI{75}{\%} over the full set of instances. While reductions
are stronger on sparse instances, significant reductions
are observed among all classes of instances. In terms of time to
optimality, the stronger relaxation pays off only for sparse instances.
As the instances become denser, the \pqformulation achieves better
running times in the shifted geometric mean. Over all instances,
however, the shifted geometric mean is reduced from 37.1 to 25.0
seconds. 
A significant portion of this improvement comes from
instances with 20 Haverly networks and only 20 additional edges. 
From the 10 instances of this class, only one instance is solved within the
time limit (in 940 seconds) by the \pqformulation while all but one are
solved using the \CJJJrelaxation. The dual bound is exactly the
problem for the \pqformulation on these instances. The approach with the \pqformulation found
an optimal solution always within the first 184 seconds of the
optimization and then used a massive amount of branching nodes to
close the gap. Overall, \CJJJ solves 11 instances more within the
timelimit than the \pqformulation.


\section{Conclusions}\label{sec:conclusions}

We have derived new valid inequalities for the pooling problem by studying a set defined by a single product, a single
pool, and a single attribute, and performing a variable aggregation. Since we have also shown these inequalities define
the convex hull in many cases, further improvements to the relaxation of the pooling problem will need to consider more
aspects of the problem. For example, still with a fixed attribute $k$, output $j$, and pool $\ell$, one may consider 
studying valid inequalities for a set in which the variables $x_{ij}$, $w_{i\ell j}$, and $q_{i\ell}$, for $i \in I$ are included, rather than being
summarized in the variables $z_{i\ell}$, $t_{k\ell}$ and $u_{k\ell j}$. Alternatively, one may still use these summary
variables, but study a set that includes multiple pools. The latter approach may yield further improved relaxations,
since it avoids the need to treat all flows to the fixed product that pass through pools other than the fixed pool as
by-pass flows.


\section*{Acknowledgements}
The authors thank Stefan Vigerske for his help and suggestions on the
implementation of the separation procedure in the GAMS environment.
Jonas Schweiger thanks the DFG for their support within Projects B06
and Z01 in CRC TRR 154 and the Research Campus MODAL funded by the
German Federal Ministry of Education and Research (BMBF) (fund number
05M14ZAM). The work of Linderoth and Luedtke was supported by the U.S.
Department of Energy, Office of Science, Office of Advanced Scientific
Computing Research, Applied Mathematics program under Contract Number
DE-AC02-06CH11357.

\bibliographystyle{siamplain}
\bibliography{opt}

\end{document}